\documentclass[a4paper]{article}

\usepackage{amsmath}
\usepackage{graphicx}
\usepackage{color}
\usepackage{multirow}
\usepackage{subfigure}
\graphicspath{{./figure/}}
\usepackage{setspace}
\usepackage{latexsym}
\usepackage{amsfonts}
\usepackage{amsthm}

\usepackage[numbers,sort&compress]{natbib}
\usepackage[colorlinks,
            linkcolor=blue,
            anchorcolor=blue,
           citecolor=blue
            ]{hyperref}

\def\mybibliography{
   \list
   {$[{\arabic{enumi}}]$}{\settowidth\labelwidth{[99]}
   \leftmargin\labelwidth \advance\leftmargin\labelsep
   \usecounter{enumi}}
   \def\newblock{\hskip .11em plus .07em minus -.07em}
   \sloppy \sfcode`\.=1000\relax }
   
\makeatletter
\def\@cite#1{\text{[{#1}]}}
\makeatother

\newcommand\abs[1]{\left \| #1\right \|}

\def\p{\boldsymbol{p}}
\def\q{\boldsymbol{q}}

\def\x{\boldsymbol{x}}

\def\z{\boldsymbol{z}}
\def\f{\boldsymbol{f}}
\def\F{\boldsymbol{F}}

\def\r{\boldsymbol{r}}

\def\v{\boldsymbol{v}}

\def\u{\boldsymbol{u}}

\def\z{\boldsymbol{z}}

\def\1{\boldsymbol{1}}

\def\0{\boldsymbol{0_3}}

\begin{document}

\title{Neighboring Optimal  Guidance for Low-Thrust Multi-Burn Orbital Transfers}


\author{\it
Zheng Chen\thanks{Ph.D Student, Laboratoire de Math\'ematiques d'Orsay, Univ. Paris-Sud, CNRS, Universit\'e Paris-Saclay, 91405, France. \underline{zheng.chen@math.u-psud.fr.}},\\
{\it University Paris-Sud $\&$ CNRS, Il-de-France, France, 91405}
}

\maketitle{}

\newcommand{\eqnref}[1]{(\ref{#1})}
 \newcommand{\class}[1]{\texttt{#1}}
 \newcommand{\package}[1]{\texttt{#1}}
 \newcommand{\file}[1]{\texttt{#1}}
 \newcommand{\BibTeX}{\textsc{Bib}\TeX}

\newtheorem{definition}{Definition}
\newtheorem{proposition}{Proposition}
\newtheorem{problem}{Problem}
\newtheorem{remark}{Remark}
\newtheorem{assumption}{Assumption}
\newtheorem{hypothesis}{Hypothesis}
\newtheorem{conjecture}{Conjecture}
\newtheorem{theorem}{Theorem}
\newtheorem{corollary}{Corollary}
\newtheorem{lemma}{Lemma}
\newtheorem{condition}{Condition}

\begin{abstract}
This paper presents a novel neighboring extremal approach to establish the neighboring optimal guidance (NOG) strategy for fixed-time low-thrust multi-burn orbital transfer problems. 
Unlike the classical variational methods which define and solve an accessory minimum problem (AMP) to design the NOG, the core of the proposed method is to construct a parameterized family of neighboring extremals around a nominal one. A geometric analysis on the projection behavior of the parameterized neighboring extremals shows that it is impossible to establish the NOG unless not only the typical Jacobi condition (JC) between switching times but also a transversal condition (TC) at each switching time is satisfied. According to the theory of field of extremals, the JC and the TC, once satisfied, are also sufficient to ensure  a multi-burn extremal trajectory to be locally optimal. Then, through deriving the first-order Taylor expansion of the parameterized neighboring extremals, the neighboring optimal feedbacks on thrust direction and switching times  are obtained.  Finally,  to verify the development of this paper, a fixed-time low-thrust fuel-optimal orbital transfer problem is calculated.
\end{abstract}



\section{Introduction}\label{SE:Introduction}

Due to numerous perturbations and errors, one cannot expect a spacecraft steered by the precomputed optimal control to exactly move on the correspondingly precomputed optimal trajectory. The precomputed optimal trajectory and  control are generally referred to as the nominal trajectory and  control, respectively. Once a  deviation from the nominal trajectory is measured by navigational systems, a guidance strategy is usually required to calculate a new (or corrected) control  in each guidance cycle such that the spacecraft can be steered by the new control to track the nominal trajectory or to move on a new optimal trajectory \cite{Lu:94}. Since the 1960s, various guidance schemes have been developed \cite{Kelley:64,Kelley:62,Chuang:96,Lee:65,Lu:08,Lu:10,Baldwin:12,Lu:03,Calise:98}, among of which there are two main categories: implicit one and explicit one. While the implicit guidance strategy generally compares the measured state  with the nominal one to generate control corrections; the explicit guidance strategy recomputes a  flight trajectory by onboard computers during its motion.  To implement an explicit guidance strategy, numerical integrations and iterations are usually required to solve a highly nonlinear two-point boundary-value problem  (TPBVP) and the time required for convergence  heavily depends on the merits of initial guesses as well as on the integration time of each iteration. In recent years, through employing a multiple shooting method and the analytical property arizing from the assumption that the gravity field is linear~\cite{Jezewski:72}, an explicit closed-loop guidance  is well developed by Lu {\it et al.} for exo-atmospheric ascent flights~\cite{Lu:08} and for deorbit problems~\cite{Baldwin:12}. This explicit type of guidance for endo-atmospheric ascent flights were studied as well in Refs.~\cite{Lu:10,Lu:03,Calise:98}. Whereas, the duration of a low-thrust orbtial transfer is so exponentially long that the onboard computer can merely afford the large amount of computational time for integrations and iterations once a shooting method is employed, which makes the explicit guidance  strategy unattractive to low-thrust orbital transfer problems.

The NOG is an implicit and less demanding guidance scheme, which not only allows the  onboard computer to realize an online computation once the gain matrices associated with the nominal extremal are computed offline and stored  in the onboard computer but also handles disturbances well \cite{Naidu:94}. Assuming the optimal control function is totally continuous, the linear feedback of control was proposed independently by Breakwell {\it et al.} \cite{Breakwell:63}, Kelley \cite{Kelley:62,Kelley:64}, Lee \cite{Lee:65},  Speyer {\it et al.} \cite{Speyer:68}, Bryson {\it et al}. \cite{Bryson:69}, and Hull \cite{Hull:03} through minimizing the second variation of the cost functional -- AMP  -- subject to the variational state and adjoint equations. Based on this method, an increasing number of literatures, including Refs.~\cite{Afshari:09,Pesch:80,Shafieenejad:13,Naidu:93,Seywald:94} and the references therein, on the topic of the NOG for orbital transfer problems have been published. 
More recently, a variable-time-domain NOG was proposed by Pontani {\it et al.} \cite{Pontani:15:1,Pontani:15:2} to avoid the numerical difficulties arising from the singularity of the gain matrices while approaching the final time and it was then applied to a continuous thrust space trajectories \cite{Pontani:15:3}.

 However,  difficulties arize when we consider to minimize the fuel consumption for a low-thrust orbital transfer because the corresponding optimal control function exhibits a bang-bang behavior  if the prescribed transfer time is bigger than the minimum transfer time for the same boundary conditions~\cite{Gergaud:06}. Considering the control function as a discontinuous scalar, the corresponding neighboring optimal feedback control law was studied by Mcintyre \cite{Mcintyre:66} and  Mcneal \cite{Mcneal:67}. Then, Foerster {\it et al.} \cite{Foerster:71} extended the work of Mcintyre and Mcneal to problems with discontinuous vector control functions. Using a multiple shooting technique, the algorithm for computing the NOG of general optimal control problems with discontinuous control and state constraints was developed in Ref.~\cite{Kugelmann:90:1}, which was then applied to a space shuttle guidance in Ref.~\cite{Kugelmann:90:2}.  As far as the author knows, a few scholars, including Chuang {\it et al.} \cite{Chuang:96} and Kornhauser {\it et al.} \cite{Kornhauser:72}, have made efforts on developing the NOG for low-thrust multi-burn orbital transfer problems. In the work \cite{Chuang:96} by Chuang  {\it et al.},  without taking into account the feedback on thrust-on times, the second variation on each burn arc was minimized such that the neighboring optimal feedbacks on  thrust direction and thrust off-times were obtained. Considering both endpoints are fixed, Kornhauser and Lion \cite{Kornhauser:72} developed an AMP for bounded-thrust optimal orbital transfer problems. Then, through minimizing this AMP, the linear feedback forms of thrust direction and switching times were derived. As is well known, it is impossible to construct the NOG unless the JC  holds along the nominal extremal \cite{Lee:65} since the gain matrices are unbounded if the JC is violated. This result was actually obtained by Kelley \cite{Kelley:62}, Kornhauser {\it et al.} \cite{Kornhauser:72}, Chuang {\it et al.} \cite{Chuang:96}, Pontani {\it et al.} \cite{Pontani:15:1,Pontani:15:2}, and many others who minimize the AMP to construct the NOG.  As a matter of fact, given every infinitesimal deviation from the nominal state, the JC, once satisfied, guarantees that there exists a neighboring extremal trajectory passing through the deviated state. Therefore, the existence of neighboring extremals is a prerequisite to establish the NOG. Once the optimal control function exhibits a bang-bang behavior, it is however not clear what conditions have to be satisfied in order to guarantee the existence of neighboring extremals \cite{Kornhauser:72}.
 
To construct the conditions that, once satisfied, guarantee that for every state in an infinitesimal neighborhood of the nominal one there exists a neighboring extremal passing through it, this paper presents a novel neighboring extremal approach to establish the NOG. The crucial idea is to construct a parameterized family of neighboring extremals around the nominal one. Then, as a result of a geometric study on the projection of the parameterized family  from tangent bundle onto state space, it is presented in this paper that the conditions sufficient for the existence of neighboring extremals around a bang-bang extremal consist of  not only  the JC between switching times  but also  a TC \cite{Schattler:12,Noble:02,Caillau:15} at each switching time. According to recent advances in geometric optimal control~\cite{Schattler:12,Caillau:15,Chen:153bp,Agrachev:04},  the JC and the TC, once satisfied, are also sufficient to guarantee the nominal extremal to be locally optimal provided some regularity assumptions are satisfied.  Given these two existence conditions, the neighboring optimal feedbacks on thrust direction and  switching times are established in this paper through deriving the first-order Taylor expansion of the parameterized neighboring extremals. 

The present paper is organized as follows: In Sect.~\ref{SE:Problem_Formulation}, the fixed-time low-thrust fuel-optimal orbital transfer problem is formulated and the first-order necessary conditions are presented by applying the Pontryagin Maximum Principle (PMP). In Sect.~\ref{SE:Sufficient}, a parameterized family of neighbouring extremals around a nominal one is first constructed.  Through analyzing the projection behavior of the parameterized family from tangent bundle onto state space, two conditions sufficient for the existence of neighboring extremals are constructed. Then, the neighboring optimal feedbacks on thrust direction and  switching times are derived. In Sect.~\ref{SE:Implementation}, the numerical implementation for the NOG scheme is presented. In Sect.~\ref{SE:Numerical}, a fixed-time low-thrust fuel-optimal orbital transfer problem is computed to verify the development of this paper. Finally, a conclusion is given in Sect. \ref{SE:Conclusion}.

\section{Optimal control problem}\label{SE:Problem_Formulation}

Throughout the paper, we denote the space of $n$-dimensional column vectors by $\mathbb{R}^n$ and the space of $n$-dimensional row vectors by $(\mathbb{R}^n)^*$.

\subsection{Dynamics}

Consider the spacecraft is controlled by a low-thrust propulsion system, the state $\x=[\r^T\ \v^T\ m]^T\in\mathbb{R}^n$ ($n=7$) for its translational motion in an Earth-centred inertial Cartesian coordinate frame (notated as $OXYZ$) consists of the position vector $\r \in \mathbb{R}^3\backslash\{0\}$, the velocity vector $\v \in \mathbb{R}^3$, and the mass $m\in\mathbb{R}_+$. Then, denote by $t\in\mathbb{R}$ the time, the set of differential equations for low-thrust orbital transfer problems  can be written as
\begin{eqnarray}
\begin{cases}
\dot{\r}(t) = \v(t),\\
\dot{\v}(t) = -\frac{\mu}{\abs{ \r(t) }^3}\r(t) + \frac{\u(t)}{m(t)},\\
\dot{m}(t) = -\beta{\abs{ \u(t)}},
\end{cases}
\label{EQ:Sigma}
\end{eqnarray}
where $\mu > 0$ is the Earth gravitational constant, the notation ``~$\abs{\cdot}$~''  denotes the Euclidean norm, $\beta > 0$ is a scalar constant determined by the specific impulse of the low-thrust engine equipped on the spacecraft, and $\u\in\mathbb{R}^3$ is the thrust (or control) vector, taking values in the admissible set
\begin{equation}
\mathcal{U}=\big\{\boldsymbol{u} \in \mathbb{R}^3 \ \arrowvert\ \abs{\boldsymbol{u} } \leq u_{max}\big\},\nonumber
\end{equation}
where the constant $u_{max} > 0$ denotes the maximum magnitude of the thrust.  Denote by $\rho\in[0,1]$  the normalized mass flow rate of the engine, i.e., $\rho = \parallel \boldsymbol{u}\parallel/u_{max}$, and let $\boldsymbol{\tau}\in\mathbb{S}^2$ be the unit vector of the thrust direction, one immediately gets $\u = u_{max} \rho \boldsymbol{\tau}$. Accordingly, $\rho$ and $\boldsymbol{\tau}$ can be considered as control variables. Set $\mathcal{T} := [0,1]\times\mathbb{S}^2$, we say $\mathcal{T}$ is the admissible set for the control $(\rho,\boldsymbol{\tau})$. Denote by the constants $m_c > 0$ and  $r_c>0$ the mass of the spacecraft without any fuel and the radius of the Earth, respectively, we define by 
\begin{eqnarray}
\mathcal{X}=\{(\r,\v,m)\in\mathbb{R}^3\backslash\{0\} \times \mathbb{R}^3\times\mathbb{R}\ \arrowvert\  \|\r\|>r_c,\ \r\times\v \neq 0,\ m \geq m_c\},\nonumber
\end{eqnarray}
the admissible set for the state $\x$. For the sake of notational clarity, let us define a controlled vector field $\f$ on $\mathcal{X}\times\mathcal{T}$ as
 \begin{eqnarray}
 \f:\mathcal{X}\times\mathcal{T}\rightarrow T_{\x}\mathcal{X},\ \f(\x,\rho,\boldsymbol{\tau})=\f_0(\x) + \rho \f_1(\x,\boldsymbol{\tau}),\nonumber
 \end{eqnarray}
 where
\begin{eqnarray}
\f_0(\x) = \left(\begin{array}{c}
\v\\
-\frac{\mu}{\parallel \r\parallel^3}\r\\
0\end{array}\right),\ \text{and}\ \f_1(\x,\boldsymbol{\tau}) = \left(\begin{array}{c}
\boldsymbol{0}\\
\frac{u_{max}}{m}\boldsymbol{\tau}\\
-\beta {u_{max}}
\end{array}\right).\nonumber
\end{eqnarray}
 Then, the dynamics in Eq.~(\ref{EQ:Sigma}) can be rewritten as
\begin{eqnarray}
{\Sigma}:
\dot{\x}(t) =\f(\x(t),\rho(t),\boldsymbol{\tau}(t)) = \f_0(\x(t)) + \rho(t) \f_1(\x(t),\boldsymbol{\tau}(t)),\ \ t\in\mathbb{R}.
\label{EQ:system}
\end{eqnarray}
Note that many mechanical systems can be represented as this control-affine form of dynamics. Thus, the NOG scheme established later can be directly applied to some other mechanical systems.

\subsection{Fuel-optimal problem}
Let $l\in\mathbb{N}$ be a finite positive integer such that $0<l\leq n$, we define the $l$-codimensional submanifold
\begin{eqnarray}
\mathcal{M}=\{\x \in\mathcal{X}\ \arrowvert\ \phi(\x)=0\}
\label{EQ:final_manifold}
 \end{eqnarray}
as the constraint submanifold for final states,  where $\phi:\mathcal{X}\rightarrow \mathbb{R}^l$ is a twice continuously differentiable function  and its expression depends on specific mission requirements.
\begin{definition}[\it Fuel-optimal problem (FOP)]
Given a fixed final time $t_f > 0$ and a fixed initial point $\x_0\in\mathcal{X}\backslash\mathcal{M}$, the fuel-optimal orbital transfer problem consists of steering the system $\Sigma$ in $\mathcal{X}$ by a measurable control $(\rho(\cdot),\boldsymbol{\tau}(\cdot)):[0,t_f]\rightarrow\mathcal{T}$ from the fixed initial point $\x_0$  to a final point $\x_f\in\mathcal{M}$ such that  the fuel consumption is minimized, i.e., 
\begin{eqnarray}
\int_{0}^{t_f}\rho(t) dt \rightarrow \text{min}.
\label{EQ:cost_functional}
\end{eqnarray}
\end{definition}
\noindent For every $u_{max}>0$, if $m_c>0$ is small enough, the controllability of the system $\Sigma$ holds in the admissible set $\mathcal{X}$ \cite{Chen:15controllability}. Let $t_m$ be the minimum transfer time of the system $\Sigma$ for the same boundary conditions as the {\it FOP}, if $t_f \geq t_m$, there exists at least one fuel-optimal solution in $\mathcal{X}$ \cite{Gergaud:06}. Thanks to the controllability and the existence results, the PMP is applicable to formulate the following necessary conditions.

\subsection{Necessary conditions}

Hereafter, we define by the column vector $\p\in \mathbb{R}^n$ the costate of $\x$.  Then, according to the PMP in Ref.~\cite{Pontryagin},  if an admissible controlled trajectory $\bar{\x}(\cdot):[0,t_f]\rightarrow \mathcal{X}$ associated with a measurable control $(\bar{\rho}(\cdot),\bar{\boldsymbol{\tau}}(\cdot)):[0,t_f]\rightarrow\mathcal{T}$ is an optimal one of the {\it FOP}, there exists a nonpositive real number $\bar{p}^0$ and an absolutely continuous mapping $t\mapsto\bar{\p}^T(\cdot)\in T^*_{\x(\cdot)}\mathcal{X}$ on $[0,t_f]$, satisfying $(\bar{\p}^T(t),\bar{p}^0)\neq 0$ for $t\in[0,t_f]$, such that  the 5-tuple $t\mapsto(\bar{\x}(\cdot),\bar{\p}^T(\cdot),\bar{p}^0,\bar{\rho}(\cdot),\bar{\boldsymbol{\tau}}(\cdot))\in T^*\mathcal{X}\times\mathbb{R}\times\mathcal{T}$ on $[0,t_f]$  is a solution of the canonical differential equations
\begin{eqnarray}
\begin{cases}
\dot{{\x}}(t) = \frac{\partial h}{\partial \p}({\x}(t),{\p}(t),{p}^0,{\rho}(t),{\boldsymbol{\tau}}(t)),\\
\dot{{\p}}(t) = -\frac{\partial h}{\partial \x}({\x}(t),{\p}(t),{p}^0,{\rho}(t),{\boldsymbol{\tau}}(t))),
\end{cases}
\label{EQ:canonical}
\end{eqnarray}
with the maximum condition
\begin{eqnarray}
h(\bar{\x}(t),\bar{\p}(t),\bar{p}^0,\bar{\rho}(t),\bar{\boldsymbol{\tau}}(t)) =\underset{(\rho(t),\boldsymbol{\tau}(t))\in\mathcal{T}}{\mathrm{max}}\ h(\bar{\x}(t),\bar{\p}(t),\bar{p}^0,\rho(t),\boldsymbol{\tau}(t))
\label{EQ:maximum_condition}
\end{eqnarray}
and the transversality condition $\bar{\p}(t_f)\perp T_{\bar{\x}(t_f)}\mathcal{M}$ being satisfied, where 
\begin{eqnarray}
h({\x},{\p},{p}^0,{\rho},{\boldsymbol{\tau}}) = {\p}^T\left[\f_0({\x}) + {\rho} \f_1({\x},{\boldsymbol{\tau}})\right] + {p}^0{\rho}
\label{EQ:Hamiltonian_h}
\end{eqnarray}
is the Hamiltonian.
Note that the transversality condition asserts 
\begin{eqnarray}
\bar{\boldsymbol{p}}^T(t_f) = \bar{\boldsymbol{\nu}} \nabla{ \phi(\bar{\x}(t_f))}  ,
\label{EQ:Transversality}
\end{eqnarray}
where $\bar{\boldsymbol{\nu}}\in(\mathbb{R}^l)^*$ is a constant vector, whose elements are Lagrangian multipliers.

Every 5-tuple $t\mapsto (\x(\cdot),\p^T(\cdot),p^0,\rho(\cdot),\boldsymbol{\tau}(\cdot))$ on $[0,t_f]$, if satisfying Eqs.~(\ref{EQ:canonical}--\ref{EQ:Hamiltonian_h}), is called an extremal. Furthermore, an extremal is called a normal one if $p^0 < 0$ and it is called an abnormal one if $p^0 = 0$. The abnormal extremals were readily ruled out by Gergaud and Haberkorn \cite{Gergaud:06}.  Thus, only normal extremals are considered and $(\p^T,p^0)$ is normalized in such a way that $p^0 = -1$ in this paper.
 According to the maximum condition in Eq.~(\ref{EQ:maximum_condition}), the extremal control $(\rho(\cdot),\boldsymbol{\tau}(\cdot))$ is a function of $(\x(\cdot),\p(\cdot))$ on $[0,t_f]$. Thus, with some abuses of notations, we denote by $(\x(\cdot),\p(\cdot))$ on $[0,t_f]$  in tangent bundle $T\mathcal{X}$ the extremal  and by $H(\x(\cdot),\p(\cdot))$ on $[0,t_f]$ the corresponding maximized Hamiltonian. Then, $H(\x,\p)$  can be written as
 \begin{eqnarray}
 H(\x,\p) := H_0(\x,\p) + \rho(\x,\p) H_1(\x,\p),\nonumber
 \end{eqnarray}
 where $H_0(\x,\p) = \p^T\f_0(\x)$ is the non-thrust Hamiltonian and $H_1(\x,\p) = \p^T\f_1(\x,\boldsymbol{\tau}(\x,\p)) - 1$ is the switching function.

Let us define by $\p_r$, $\p_v$, and $p_m$ the costates of $\r$, $\v$, and $m$, respectively, such that $\p=[\p_r^T\ \p_v^T\ p_m]^T$. Then, the maximum condition in Eq.~(\ref{EQ:maximum_condition}) implies 
\begin{eqnarray}
\begin{cases}\rho =1,\ \ \ \ \  \text{if}\ H_1 > 0,
\\
\rho = 0, \ \ \ \ \ \text{if}\ H_1 < 0,
\end{cases}
\label{EQ:Max_condition1}
\end{eqnarray}
and
\begin{eqnarray}
\boldsymbol{\tau} = \p_v/ \abs{\p_v },\ \  \text{if}\ \ \abs{ \p_v}\neq 0.
\label{EQ:Max_condition2}
\end{eqnarray}
Thus, the optimal direction of the thrust vector $\boldsymbol{u}$ is collinear to $\p_v$, well known as the primer vector~\cite{Lawden:63}. An extremal $(\x(\cdot),\p(\cdot))\in T\mathcal{X}$ on $[0,t_f]$ is called a singular one if $H_1(\x(\cdot),\p(\cdot)) \equiv 0$ on a finite interval $[t_1,t_2]\subseteq[0,t_f]$ with $t_1 < t_2$, and the singular value of $\rho$ can be obtained by repeatedly differentiating the identity $H_1(\x,\p) \equiv 0$ until $\rho$ explicitly appears \cite{Kelley:66}.  It  is called a nonsingular one if the switching function $H_1(\x(\cdot),\p(\cdot))$ on $[0,t_f]$  has either no or only isolated zeros. 

The NOG for a totally singular extremal was studied by Breakwell and Dixon \cite{Breakwell:75}. If the thrust is continuous along a nonsingular extremal, the classical variational method \cite{Speyer:68,Bryson:69,Kelley:64,Kelley:62,Lee:65,Pontani:15:1,Pontani:15:2,Pontani:15:3} can be directly employed to design the NOG. In next section, the NOG for  bang-bang extremals will be established through constructing a parameterized family of extremals.

\section{Neighboring optimal guidance}\label{SE:Sufficient}

Hereafter, we always denote by $(\bar{\x}(\cdot),\bar{\p}(\cdot)):[0,t_f]\rightarrow T\mathcal{X}$ and $(\bar{\rho}(\cdot),\bar{\boldsymbol{\tau}}(\cdot)):[0,t_f]\rightarrow \mathcal{T}$  the nominal extremal and the associated nominal control, respectively, and we assume the nominal extremal is readily computed. 
\begin{definition}[Neighboring extremal]
Let $\mathcal{W}\subseteq T\mathcal{X}$ be a small  tubular neighbourhood of the nominal extremal $(\bar{\x}(\cdot),\bar{\p}(\cdot))$ on $[0,t_f]$, we say every time solution trajectory $(\x(\cdot),\p(\cdot)):[0,t_f]\rightarrow T\mathcal{X}$ of Eqs.~(\ref{EQ:canonical}--\ref{EQ:Transversality}) is a neighboring extremal of the nominal one if $(\x(\cdot),\p(\cdot))$ on $[0,t_f]$ lies in $\mathcal{W}$.
\label{DE:Reference}
\end{definition}
 \noindent In next paragraph, the neighboring extremals will be parameterized.
\subsection{Parameterization of neighbouring extremals}

Let us define a submanifold $\mathcal{L}_f\subset T\mathcal{X}$ as
\begin{eqnarray}
\mathcal{L}_f = \big\{(\x,\p)\in T\mathcal{X}\ \arrowvert \ \x\in\mathcal{M},\ \p\perp T_{\x}\mathcal{M}\big\}.\nonumber
\end{eqnarray}
Then, according to Definition \ref{DE:Reference}, for every neighbouring extremal $(\x(\cdot),\p(\cdot))\in \mathcal{W}$ on $[0,t_f]$, there holds $(\x(t_f),\p(t_f))\in\mathcal{L}_f$. 
Note that the submanifold $\mathcal{L}_f$ is of dimension $n$ once the matrix $\nabla \phi(\bar{\x}(t_f))$ is of full rank.
\begin{assumption}
The matrix $\nabla \phi(\bar{\x}(t_f))$ is of full rank.
\label{AS:full_rank}
\end{assumption}
\noindent As a result of this assumption, let $\mathcal{N}\subset \mathcal{L}_f$ be a sufficiently small open neighbourhood of $(\bar{\x}(t_f),\bar{\p}(t_f))$, there exists an invertible function $\boldsymbol{F}:\mathcal{N}\rightarrow (\mathbb{R}^n)^*$ such that both the function and its inverse $\F^{-1}$ are smooth. Then, for every $\q\in(\mathbb{R}^n)^*$, there exists one and only one $(\x,\p)\in\mathcal{N}$ such that $\q = \F(\x,\p)$.
  Let us define by 
\begin{eqnarray}
\gamma: [0,t_f]\times\F(\mathcal{N})\rightarrow T\mathcal{X}, \ \gamma(t,\q) = (\x(t),\p(t))\nonumber
\end{eqnarray}
the time solution trajectory of Eqs.~(\ref{EQ:canonical}--\ref{EQ:Transversality}) such that $$\boldsymbol{F}^{-1}(\q) = \gamma(t_f,\q),$$ i.e., there holds $\gamma(t_f,\q)\in\mathcal{N}$ for every $\q\in\F(\mathcal{N})$. Then,   let $\bar{\q} = \boldsymbol{F}(\bar{\x}(t_f),\bar{\p}(t_f))$, we have $(\bar{\x}(\cdot),\bar{\p}(\cdot)) = \gamma(\cdot,\bar{\q})$ on $[0,t_f]$. 
\begin{definition}
Given the nominal extremal $(\bar{\x}(\cdot),\bar{\p}(\cdot)) = \gamma(\cdot,\bar{\q})$ on $[0,t_f]$,  we denote by
\begin{eqnarray}
\mathcal{F} = \big\{(\x(t),\p(t))\in T\mathcal{X} \ \arrowvert\ (\x(t),\p(t))=\gamma(t,\q),\ t\in[0,t_f],\ \q \in\boldsymbol{F}(\mathcal{N})\big\}\nonumber
\end{eqnarray}
the $\q$-parameterized family of neighbouring extremals  around the nominal extremal $\gamma(t,\bar{\q})$ on $[0,t_f]$.
\label{DE:parameterized}
\end{definition} 
\noindent 
For the sake of notational clarity, let us define a mapping
\begin{eqnarray}
\Pi:  T\mathcal{X} \rightarrow  \mathcal{X},\  (\x,\p) \mapsto \x\nonumber
\end{eqnarray}
that projects a submanifold from the tangent bundle $T\mathcal{X}$ onto the state space $\mathcal{X}$. 
\begin{definition}[Existence of neighboring extremals]
Given the nominal extremal $(\bar{\x}(\cdot),\bar{\p}(\cdot)) = \gamma(\cdot,\bar{\q})$ on $[0,t_f]$, we say that there exist neighboring extremals around this nominal extremal if, for every $t\in[0,t_f)$ and every $\x_*\in\mathcal{X}\backslash\mathcal{M}$ in an infinitesimal neighborhood  of $\bar{\x}(t)$,  there exists a small subset $\mathcal{N}$ and $\q_*\in\F(\mathcal{N})$ such that $\x_*= \Pi(\gamma(t,\q_*))$.
\label{DE:existence} 
\end{definition}
\noindent  Note that the NOG is constructed by using the Taylor expansion of the neighboring extremal to approximate the corresponding neighboring optimal control. Thus, the existence of neighboring extremals around the nominal extremal  $\gamma(\cdot,\bar{\q})$ on $[0,t_f]$ is a prerequisite to construct the NOG \cite{Lee:65}. In next subsection, through analyzing the projection behavior of the family $\mathcal{F}$ at each time $t\in[0,t_f)$ from $T\mathcal{X}$ onto $\mathcal{X}$, the  conditions for the existence of neighboring extremals around a nominal one with a bang-bang control will be established.


\subsection{Conditions for the existence of neighbouring extremals}\label{SE:Existence}


 Hereafter, we denote by $t_0\in[0,t_f)$ the current time and let $\x_*\in\mathcal{X}\backslash\mathcal{M}$ be the measured (or actual) state of the spacecraft at $t_0$. Generally speaking, there holds 
 $$\Delta \x := \x_* - \bar{\x}(t_0) \neq 0.$$
 Let $t_m^*>0$ be the minimum time  to steer the system $\Sigma$ by measurable controls $(\rho(\cdot),\boldsymbol{\tau}(\cdot)):[0,t^*_m]\rightarrow\mathcal{T}$ from the actual state $\x_*\in\mathcal{X}\backslash\mathcal{M}$ to a point $\x_f\in\mathcal{M}$, if $t_f - t_0 < t^*_m$ for every $\x_f\in\mathcal{M}$, there is even not an admissible controlled trajectory on the time interval $[t_0,t_f]$ connecting $\x_*$ and $\mathcal{M}$ .
\begin{assumption}
There exists at least one point $\x_f\in\mathcal{M}$ such that $t_f - t_0 \geq t^*_m$.
\label{AS:tm}
\end{assumption}
\noindent According to the controllability results in Ref. \cite{Chen:15controllability},  this assumption implies that there exists at least one fuel-optimal trajectory $\hat{\x}(\cdot)\in\mathcal{X}$ on $[t_0,t_f]$ such that $\hat{\x}(t_0) = \x_*$ and $\hat{\x}(t_f)  \in\mathcal{M}$ \cite{Gergaud:06}. However, one cannot use the technique of Taylor expansion to design the NOG unless the fuel-optimal trajectory $\hat{\x}(t)$ is a neighboring extremal such that the higher order terms are negligible. In this subsection, provided that Assumption \ref{AS:tm} is satisfied, we will establish some conditions which, once satisfied, guarantee the existence of neighboring extremals (cf. Definition \ref{DE:existence}) such that the NOG can be constructed.

\begin{proposition}
Given the nominal extremal $(\bar{\x}(\cdot),\bar{\p}(\cdot))=\gamma(\cdot,\bar{\q})$ on $[0,t_f]$, let  Assumption \ref{AS:tm} be satisfied  for every $t_0\in[0,t_f)$ and denote by $\mathcal{O}_{t_0}\subset\mathcal{X}\backslash\mathcal{M}$ an infinitesimal open neighborhood of the point $\bar{\x}(t_0)$. Then, if the point $\bar{\x}(t_0)$ lies on the boundary of the domain $\Pi(\gamma(t_0,\F(\mathcal{N})))$ for a subset $\mathcal{N}\subset\mathcal{L}_f$, no matter how small the neighborhood $\mathcal{O}_{t_0}$ is, there are some $\x_*\in\mathcal{O}_{t_0}\backslash\{\bar{\x}(t_0)\}$ such that $\x_* \not\in \Pi(\gamma(t_0,\F(\mathcal{N})))$, i.e, no neighboring extremals in the family $\mathcal{F}$ restricted to the subset $\mathcal{N}$ can pass through the point $\x_*$ at $t_0$.
\label{PR:Diffeomorphism1}
\end{proposition}
\begin{proof}
If  the point $\bar{\x}(t_0)$ lies on the boundary of the domain $\Pi(\gamma(t_0,\F(\mathcal{N})))$, for every open neighborhood $\mathcal{O}_{t_0}\subset\mathcal{X}\backslash\mathcal{M}$ of $\bar{\x}(t_0)$, the set $\mathcal{O}_{t_0}\backslash(\mathcal{O}_{t_0}\cap \Pi(\gamma(t_0,\F(\mathcal{N}))))$ is not empty. Thus, for every $\x_*\in \mathcal{O}_{t_0}\backslash(\mathcal{O}_{t_0}\cap \Pi(\gamma(t_0,\F(\mathcal{N}))))$, there holds $\x_*\not \in \Pi(\gamma(t_0,\F(\mathcal{N})))$, which proves the proposition.
\end{proof}
\noindent If the projection $\Pi$ of $\mathcal{F}$ at $t_0$ is a fold singularity  \cite{Agrachev:04},  the trajectories $\x(\cdot)=\Pi(\gamma(\cdot,\q))$ around $t_0$ intersect with each other as is shown by the typical picture in Figure \ref{Fig:smooth_fold}.
\begin{figure}[!ht]
 \begin{center}
 \includegraphics[trim=1cm 2cm 1cm 0cm, clip=true, width=3.0in,angle=0]{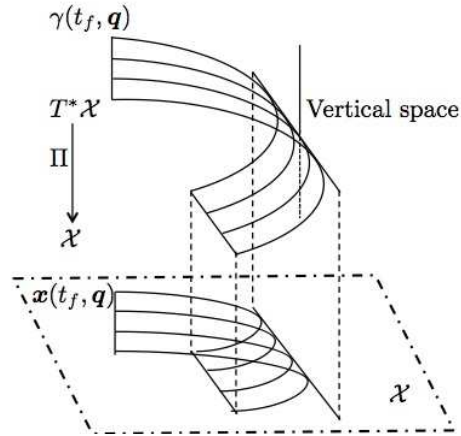}
\end{center}
 \caption[]{The fold singularity for the projection of the family $\mathcal{F}$ \cite{Agrachev:04}.}
 \label{Fig:smooth_fold}
\end{figure}
As is illustrated by the right plot in Figure \ref{Fig:smooth_fold1}, the point $\bar{\x}(t_0) = \Pi(\gamma(t_0,\bar{\q}))$ lies on the boundary of the domian $\Pi(\gamma(t_0,\F(\mathcal{N})))$ for every sufficiently small subset $\mathcal{N}\subset\mathcal{L}_f$ if the projection $\Pi$ of $\mathcal{F}$ at $t_0$ is a fold singularity.
\begin{figure}[!ht]
 \begin{center}
 \includegraphics[trim=0cm 1.8cm 0cm 0cm, clip=true, width=0.8\textwidth]{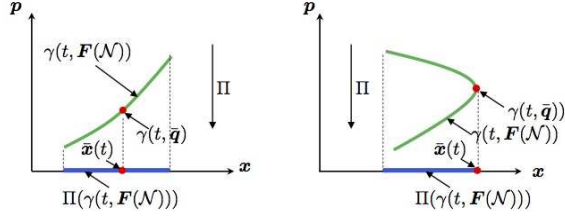}
\end{center}
 \caption[]{The section of the family $\mathcal{F}$ at a time $t\in[0,t_f)$. The projection $\Pi$ in the left plot is a diffeomorphism and the projection $\Pi$ in the right plot is a fold singularity.}
 \label{Fig:smooth_fold1}
\end{figure}
 Consequently, Proposition \ref{PR:Diffeomorphism1} indicates that for some sufficiently small deviation $\Delta \x$ there holds $\bar{\x}(t)+\Delta \x \neq \Pi(\gamma(t,\q))$ for every $\q\in\mathcal{N}$ once the projection $\Pi$ of $\mathcal{F}$ at $t\in[0,t_f)$ is a fold singularity.
\begin{proposition}
Given the nominal extremal $(\bar{\x}(\cdot),\bar{\p}(\cdot))=\gamma(\cdot,\bar{\q})$ on $[0,t_f]$, let Assumption \ref{AS:tm} be satisfied  for every $t_0\in[0,t_f)$ and denote by $\mathcal{O}_{t_0}\subset\mathcal{X}\backslash\mathcal{M}$ an infinitesimal open neighborhood of the point $\bar{\x}(t_0)$. Then, for every $\x_*\in\mathcal{O}_{t_0}$, there exists a $\q_*\in\F(\mathcal{N})$ such that $\x_* = \Pi(\gamma(t_0,\q_*))$  if the projection $\Pi$ of $\mathcal{F}$ at $t_0$ is a diffeomorphism.
\label{PR:Diffeomorphism}
\end{proposition}
\begin{proof}
If the projection $\Pi$ of $\mathcal{F}$ at $t_0$ is a diffeomorphsim as is shown by the left plot in Figure \ref{Fig:smooth_fold1},  the mapping $\q\mapsto\Pi(\gamma(t_0,\q))$ from the domain $\F(\mathcal{N})$ onto its image is a homeomorphism. 
Note that the subset $\F(\mathcal{N})$ is an open neighborhood of $\bar{\q}$. Thus, under the hypotheses of this proposition, the image $\Pi(\gamma(t_0,\F(\mathcal{N})))$ is an open neighborhood of $\bar{\x}(t_0) = \Pi(\gamma(t_0,\bar{\q}))$ according to the {\it inverse function theorem}.
 Then, there exists a sufficiently small neighborhood $\mathcal{O}_{t_0}\subset\mathcal{X}\backslash\mathcal{M}$ of $\bar{\x}(t_0)$ such that $\mathcal{O}_{t_0}\subset \Pi(\gamma(t_0,\F(\mathcal{N})))$. Referring to the {\it inverse function theorem} again, for every $\x_*\in\mathcal{O}_{t_0}$, there exists one and only one $\q_*\in\F(\mathcal{N})$ such that $\x_* = \Pi(\gamma(t_0,\q_*))$, which proves the proposition.
\end{proof}
\noindent   Note that the projection $\Pi$ of $\mathcal{F}$ loses its local diffeomorphism if it is a fold singularity. Thus, as a combination of Definition \ref{DE:existence} and Propositions \ref{PR:Diffeomorphism1}  and \ref{PR:Diffeomorphism}, to formulate the conditions  for the existence of neighboring extremals around $(\bar{\x}(\cdot),\bar{\p}(\cdot))$ on $[t_0,t_f]$, it is enough to establish the conditions that guarantee the projection $\Pi$ of the family $\mathcal{F}$ at each time $t\in[t_0,t_f)$   is a diffeomorphism. In next paragraph, the conditions related to the projection properties of $\mathcal{F}$ at each time $t\in[0,t_f)$ will be established.

 Without loss of generality, we assume that, from the current time $t_0$ on, there exist $k\in\mathbb{N}$ switching times ${t}_i$ ($i=1,2,\cdots,k$) such that $t_0 < {t}_1<{t}_2<\cdots<{t}_k<t_f$ along the nominal extremal $(\bar{\x}(\cdot),\bar{\p}(\cdot))$ on $[t_0,t_f]$. 
\begin{assumption}
Along the nominal extremal $(\bar{\x}(\cdot),\bar{\p}(\cdot))$ on $[t_0,t_f]$, each switching point $(\bar{\x}({t}_i),\bar{\p}({t}_i))$ is assumed to be a regular one, i.e., ${H}_1(\bar{\x}({t}_i),\bar{\p}({t}_i)) = 0$ and $\dot{H}_1(\bar{\x}({t}_i),\bar{\p}({t}_i)) \neq 0$ for $i=1,2,\cdots,k$.
\label{AS:regular_switching}
\end{assumption}
\noindent As a result of this assumption, if the subset $\mathcal{N}$ is small enough, the $i$-th switching time of the extremals $\gamma(\cdot,\q)$ in $\mathcal{F}$ is a smooth function of $\q$. Thus, we are able to define
\begin{eqnarray}
t_i:\F(\mathcal{N})\rightarrow \mathbb{R},\ \q\mapsto t_i(\q),
\label{EQ:switching_time}
\end{eqnarray}
as the $i$-th switching time of the extremal $\gamma(\cdot,\q)$ on $[t_0,t_f]$ for $\q\in\F(\mathcal{N})$ \cite{Caillau:15}. 

Set 
$$(\x(t,\q),\p(t,\q)) := \gamma(t,\q),\ (t,\q)\in[0,t_f]\times\F(\mathcal{N}).$$
 If the matrix $\frac{\partial \x}{\partial \q}(t,\bar{\q})$ is singular at a time $t\in(t_i,t_{i+1})$, the projection $\Pi$ of the family $\mathcal{F}$ at $t$ is a fold singularity \cite{Agrachev:04,Caillau:15,Chen:153bp,Schattler:12}. 
\begin{condition}
The matrix $\frac{\partial \x}{\partial \q}(\cdot,\bar{\q})$ is invertible on $[t_0,t_f)$, i.e., $\det\left[\frac{\partial \x}{\partial \q}(\cdot,\bar{\q})\right] \neq 0$ on $[t_0,t_f)$.
\label{CON:disconjugacy_bang}
\end{condition}
\noindent This condition is equivalent with the JC \cite{Agrachev:04}. If the subset $\mathcal{N}$ is small enough, this condition guarantees that projection $\Pi$ of the family $\mathcal{F}$ on each subinterval $(t_i,t_{i+1})$ for $i=0,1,\cdots,k$ with $ t_{k+1} = t_f$  is a diffeomorphism, see Refs.~\cite{Caillau:15,Chen:153bp,Schattler:12}. However, this condition is not sufficient to guarantee the  projection $\Pi$ of the family $\mathcal{F}$ on the whole semi-open interval $[t_0,t_f)$ is a diffeomorphism  because there exists another type of fold singularity  near each switching time $t_i$, as is illustrated by Figure \ref{Fig:trans} that the trajectories $\x(t,\q)$ around the switching time $t_i(\q)$ may intersect with each other \cite{Noble:02,Schattler:12}.
 \begin{figure}[!ht]
 \begin{center}
 \includegraphics[trim=1cm 2cm 1cm 0cm, clip = true, width=3.5in]{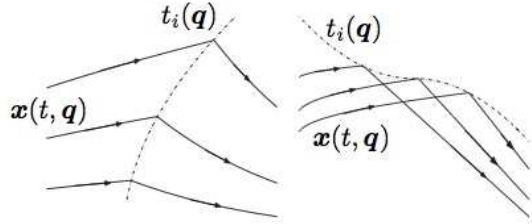}
 \end{center}
 \caption[]{The left plot shows that the projection $\Pi$ of $\mathcal{F}$ is a diffeomorphism around a switching time $t_i(\q)$ and the right plot shows that the projection  $\Pi$ of $\mathcal{F}$ is a fold singularity around a switching time $t_i(\q)$ \cite{Noble:02,Schattler:12}.}
 \label{Fig:trans}
\end{figure} 

Let $t_i^{-}(\q)$ and $t_i^+(\q)$ denote the instants a priori to and after the switching time $t_i(\q)$, respectively. If Condition \ref{CON:disconjugacy_bang} is satisfied, according to Refs.~\cite{Caillau:15,Noble:02},  there exists an inverse function $(t,\x)\mapsto\q(t,\x)$ such that
\begin{eqnarray}
 \psi_i(t,\x) := t - t_i(\q(t,\x)) = 0.\nonumber
\end{eqnarray}
Then, the set 
\begin{eqnarray}
\mathcal{S}_i:=\big\{(t,\x)\ \arrowvert \ \psi_i(t,\x) = 0\big\}\nonumber
\end{eqnarray}
is the switching surface in $(t,\x)$-space.
Obviously, the projection $\Pi$ of $\mathcal{F}$ at the switching time $t_i$ is a diffeomorphism if the flows $(t,\x(t,\q))$ on the sufficiently short interval $[t_i(\q)-\sigma,t_i(\q)+\sigma]$ with $\sigma > 0$ cross the switching surface  $\mathcal{S}_i$  transversally~\cite{Noble:02,Schattler:12,Caillau:15}. As is shown by the left plot in Figure \ref{Fig:trans}, this transversally crossing means that, for every $\q\in\F(\mathcal{N})$, the tangent vectors $\boldsymbol{T}_i^{\pm}(\q)\in(\mathbb{R}^{n+1})^*$ of the flows $(t_i^{\pm}(\q),\x^T(t_i^{\pm}(\q),\q))$, i.e.,
 $$\boldsymbol{T}_i^{\pm}(\q) := (1,\dot{\x}^T(t_i^{\pm}(\q),\q)),$$
  point to the same side of the switching surface $\mathcal{S}_i$, i.e.,
\begin{eqnarray}
\left[\boldsymbol{T}_i^{+}(\q)\boldsymbol{N}_i(\q)\right]\left[\boldsymbol{T}_i^{-}(\q)\boldsymbol{N}_i(\q)\right]>0,
\label{EQ:TC}
\end{eqnarray}
 where the vector $\boldsymbol{N}_i:\F(\mathcal{N})\rightarrow\mathbb{R}^{n+1},\ \boldsymbol{N}_i(\q) = {\nabla \psi_i(t_i(\q),\x(t_i(\q),\q))}$ denotes the normal vector of the switching surface $\mathcal{S}_i$ at $(t_i(\q),\x(t_i(\q),\q))$. In Refs.~\cite{Noble:02,Schattler:12,Caillau:15}, Eq.~(\ref{EQ:TC}) is called as the TC. In contrast,  the projection $\Pi$ of $\mathcal{F}$ at the switching time $t_i$ is a fold singularity if the tangent vectors $\boldsymbol{T}_i^{\pm}(\q)$ point to the two different sides of the switching surface $\mathcal{S}_i$, i.e., 
  \begin{eqnarray}
\left[\boldsymbol{T}_i^{+}(\q)\boldsymbol{N}_i(\q)\right]\left[\boldsymbol{T}_i^{-}(\q)\boldsymbol{N}_i(\q)\right]<0.
\label{EQ:FC}
\end{eqnarray}
According to Theorem 2 in Ref.~\cite{Caillau:15}, the computation of Eq.~(\ref{EQ:TC}) and Eq.~(\ref{EQ:FC}) can be reduced to testing the sign property of $\det\left[\frac{\partial \x}{\partial \q}(t,\q)\right]$, as is presented by the following remark.
\begin{remark}[Chen {\it et al.}~\cite{Caillau:15}]
Assume the subset $\mathcal{N}$ is small enough and that Assumption \ref{AS:regular_switching} is satisfied. Then, for every $\q \in \F(\mathcal{N})$, Eq.~(\ref{EQ:TC}) is satisfied if and only if 
\begin{eqnarray}
\text{det}\left[\frac{\partial \x}{\partial \q}(t_{i}^+({\q}),{\q})\right]\text{det}\left[\frac{\partial \x}{\partial \q}(t_{i}^-({\q}),{\q})\right] > 0.\nonumber
\end{eqnarray}
 And, Eq.~(\ref{EQ:FC}) is satisfied if  and only if
\begin{eqnarray}
\text{det}\left[\frac{\partial \x}{\partial \q}(t_{i}^+(\q),{\q})\right]\text{det}\left[\frac{\partial \x}{\partial \q}(t_{i}^-(\q),{\q})\right] < 0.
\label{EQ:FC1}
\end{eqnarray}
\end{remark}
 \begin{condition}
Let the strict inequality $\text{det}\left[\frac{\partial \x}{\partial \q}(t_{i}^+(\bar{\q}),\bar{\q})\right]\text{det}\left[\frac{\partial \x}{\partial \q}(t_{i}^-(\bar{\q}),\bar{\q})\right] > 0$ be satisfied 
at each switching time ${t}_i$ ($i=1,2,\cdots,k$) of the nominal extremal $(\bar{\x}(\cdot),\bar{\p}(\cdot))$ on $[t_0,t_f]$.
\label{AS:Transversal}
\end{condition} 
\noindent 
According to previous analysis, if Assumption \ref{AS:regular_switching} is satisfied and the subset $\mathcal{N}$ is small enough, Conditions \ref{CON:disconjugacy_bang} and \ref{AS:Transversal} are sufficient to guarantee the projection $\Pi$ of $\mathcal{F}$ at each time $t\in[t_0,t_f)$ is a diffeomorphism. Then, according to  Proposition \ref{PR:Diffeomorphism}, one obtains the following result.
 \begin{corollary}
Given the nominal extremal $(\bar{\x}(\cdot),\bar{\p}(\cdot))$ on $[t_0,t_f]$ such that each switching point is regular (cf. Assumption \ref{AS:regular_switching}), let {\it Assumption} \ref{AS:tm} be satisfied for every $t_0\in[0,t_f)$. Then, for every measured state $\x_*\in\mathcal{X}\backslash\mathcal{M}$ in an infinitesimal neighborhood of $\bar{\x}(t)$,  there exists a $\q_*\in\F(\mathcal{N})$ such that $\x_* = \x(t,\q_*)$ if {\it Conditions} \ref{CON:disconjugacy_bang} and \ref{AS:Transversal} are satisfied.
\label{CO:existence}
\end{corollary}
\noindent  Therefore, the conditions sufficient for the existence of neighboring extremals consist of not only the JC (or Condition \ref{CON:disconjugacy_bang}) between switching times but also the TC (or Condition \ref{AS:Transversal})  at each switching time once the nominal control is discontinuous. 

By applying {\it Theorem 17.2} in Ref.~\cite{Agrachev:04} or the {\it Shadow-Price Lemma} in Refs.~\cite{Noble:02,Schattler:12}, one can directly obtain the following result for optimality.
\begin{theorem}
Given the nominal extremal $(\bar{\x}(\cdot),\bar{\p}(\cdot))$ on $[t_0,t_f]$ such that each switching point is regular (cf. Assumption \ref{AS:regular_switching}), if Conditions  \ref{CON:disconjugacy_bang} and \ref{AS:Transversal} are satisfied, the nominal trajectory $\bar{\x}(\cdot)$ on $[t_0,t_f]$ realizes a  minimum cost of Eq.~(\ref{EQ:cost_functional}) with respect to every  admissible controlled trajectory $\x(\cdot)$ on $[t_0,t_f]$ in $\Pi(\mathcal{F})$ associated with the measurable control $(\rho(\cdot),\boldsymbol{\tau}(\cdot)):[t_0,t_f]\rightarrow \mathcal{T}$ with the same endpoints $\bar{\x}(t_0) = \x(t_0)$ and $\bar{\x}(t_f) = \x(t_f)$, i.e., there holds
\begin{eqnarray}
\int_{t_0}^{t_f}\bar{\rho}(t)\leq \int_{t_0}^{t_f}\rho(t)dt,\nonumber
\end{eqnarray}
where the equality holds if and only if $\bar{\x}(\cdot) = \x(\cdot)$ on $[t_0,t_f]$.
\label{TH:optimality}
\end{theorem}
\noindent Consequentely, {\it Conditions} \ref{CON:disconjugacy_bang} and \ref{AS:Transversal} are also sufficient to guarantee the local optimizer or the absence of conjugate points on the nominal extremal $(\bar{\x}(\cdot),\bar{\p}(\cdot))$ on $[t_0,t_f]$. Note that a conjugate point, beyond which the reference extremal loses its local optimality, occurs at the switching time $t_i(\q)$ of the extremal $\gamma(\cdot,\q)$ on $[t_0,t_f)$ if Eq.~(\ref{EQ:FC1}) is satisfied~\cite{Caillau:15}.
\begin{remark}
Notice that the matrix $\left[{\partial \p(t_i^{\pm},\bar{\q})}/{\partial \q}\right]\left[{\partial \x(t_i^{\pm},\bar{\q})}/{\partial \q}\right]^{-1}$ can keep bounded even though Eq.~(\ref{EQ:FC1}) is satisfied. Hence, the classical variational method \cite{Bryson:69}, which detects conjugate points through testing the unbounded time of the matrix $\left[{\partial \p(t,\bar{\q})}/{\partial \q}\right]\left[{\partial \x(t,\bar{\q})}/{\partial \q}\right]^{-1}$, fails to detect the occurrence of  conjugate points at switching times. One has to test Eq.~(\ref{EQ:FC1}) at each switching time  to see if a conjugate point occurs at the switching time.
\end{remark}


\subsection{Neighbouring optimal feedback control law}

As is explained in Sect. \ref{SE:Introduction}, a spacecraft cannot exactly move on the nominal trajectory $\bar{\x}(\cdot) = \Pi(\gamma(\cdot,\bar{\q}))$ on $[0,t_f]$.  According to Corollary \ref{CO:existence}, if Conditions \ref{CON:disconjugacy_bang} and \ref{AS:Transversal} are satisfied and the deviation $\Delta \x$ is small enough,  there then exists a  $\q_*\in\F(\mathcal{N})$ such that $\bar{\x}(t_0) + \Delta \x = \Pi(\gamma(t_0,\q_*))$. Obviously, once the new extremal $\gamma(\cdot,\q_*)$ on the interval $[t_0,t_f]$ is computed, if no further perturbations occur  for $t>t_0$, the spacecraft can be steered by the associated new optimal control function $\u(\gamma(\cdot,\q_*))$ on $[t_0,t_f]$ to fly to $\mathcal{M}$. Though various numerical methods, e.g., direct ones, indirect ones, and hybrid ones, are available in the literature to compute $\gamma(\cdot,\q_*)$ on $[t_0,t_f]$, the onboard computer can  merely afford this computation in each guidance cycle, especially for the low-thrust orbital transfer problem with a long duration. 

Next, the neighboring optimal feedback control strategy, which is the first-order Taylor expansion of the optimal control $\u(\gamma(\cdot,\q_*))$ on $[t_0,t_f]$, will be derived such that the spacecraft can be controlled to move closely enough along the extremal trajectory $\x(\cdot,\q_*)= \Pi(\gamma(\cdot,\q_*))$ on $[t_0,t_f]$ if the deviation $\Delta \x$ is small enough.

\subsubsection{Neighboring optimal feedback on switching times}

Note that  $t_i(\q_*)$ is exactly the $i$-th switching time of the new extremal $\gamma(\cdot,\q_*)$ on $[t_0,t_f]$. Set $\Delta \q := \q_* - \bar{\q}$, the first-order Taylor expansion of $t_i(\q_*)$ is
\begin{eqnarray}
\Delta t_i :=t_i(\q_*)  - t_i(\bar{\q}) = \frac{dt_i(\bar{\q})}{d\q}\Delta \q^T + O_{t_i}(\abs{\Delta\q}^2),
\label{EQ:Delta_ti}
\end{eqnarray}
where $O_{t_i}(\abs{\Delta\q}^2)$ is the sum of second and higher order terms. Note that there holds $$H_1(\x(t_i(\q),\q),\p(t_i(\q),\q)) \equiv 0$$ for every $\q\in\F(\mathcal{N})$. Differentiating the identity $H_1(\x(t_i(\q),\q),\p(t_i(\q),\q)) \equiv 0$ with respect to $\q$ yields
\begin{eqnarray}
0  &=& \frac{\partial H_1(\gamma(t_i(\q),\q))}{\partial \x^T} \left(\dot{\x}(t_i(\q),\q)\frac{dt_i(\q)}{d\q} + \frac{\partial \x(t_i(\q),\q)}{\partial \q}\right) \nonumber\\
&+& \frac{\partial H_1(\gamma(t_i(\q),\q))}{\partial \p^T} \left({\dot{ \p}}(t_i(\q),\q)\frac{dt_i(\q)}{d\q} + \frac{\partial \p(t_i(\q),\q)}{\partial \q}\right)  \nonumber\\
&=&\dot{H}_1(\gamma(t_i(\q),\q))\frac{dt_i(\q)}{d\q} + \frac{\partial H_1(\gamma(t_i(\q),\q))}{\partial \x^T}\frac{\partial \x(t_i(\q),\q)}{\partial \q}
\nonumber\\
&+& \frac{\partial H_1(\gamma(t_i(\q),\q))}{\partial \p^T}  \frac{\partial \p(t_i(\q),\q)}{\partial \q}.
 \label{EQ:variation_H1_ti}
\end{eqnarray}
Note that $\dot{H}_1(\bar{\x}(t_i),\bar{\p}(t_i)) \neq 0$ by Assumption \ref{AS:regular_switching}, one obtains
\begin{eqnarray}
\frac{d t_i(\bar{\q})}{d \q} &=& -\Big[ \frac{\partial H_1(\bar{\x}(t_i),\bar{\p}(t_i))}{\partial \x^T}\frac{\partial \x(t_i,\bar{\q})}{\partial \q}\nonumber\\
& +& \frac{\partial H_1(\bar{\x}(t_i),\bar{\p}(t_i))}{\partial \p^T}\frac{\partial \p(t_i,\bar{\q})}{\partial \q}\Big]/\dot{H}_1(\bar{\x}(t_i),\bar{\p}(t_i)),
\label{EQ:variation_switching_time}
\end{eqnarray}
where two vectors $$\frac{\partial H_1}{\partial \x^T}(\bar{\x}(t_i),\bar{\p}(t_i)) = \bar{\p}^T(t_i)\cdot \frac{\partial \f_1}{\partial \x}(\bar{\x}(t_i),\boldsymbol{\tau}(\bar{\x}(t_i),\bar{\p}(t_i))),$$ 
$$\frac{\partial H_1}{\partial \p^T}(\bar{\x}(t_i),\bar{\p}(t_i)) = \f_1^T(\bar{\x}(t_i),\boldsymbol{\tau}(\bar{\x}(t_i),\bar{\p}(t_i))),$$ can be directly computed once  the nominal extremal $(\bar{\x}(\cdot),\bar{\p}(\cdot)) = \gamma(\cdot,\bar{\q})$ on $[t_0,t_f]$ is given. 

According to Corollary \ref{CO:existence},  for every sufficiently small $\Delta \x$ and every time $t_0\in[0,t_f)$, one has the following first-order Taylor expansion
\begin{eqnarray}
\Delta \q^T &=& \q_*^T - \bar{\q}^T\nonumber\\
&=& \left[\frac{\partial \x(t_0,\bar{\q})}{\partial \q}\right]^{-1}\Delta \x+ O_{\q}(\abs{\Delta \x}^2),
\label{EQ:Delta_q}
\end{eqnarray}
where $O_{\q}(\abs{\Delta \x}^2)$ denotes the sum of second and higher order terms. 
For notational clarity, let us define a matrix-valued function $S:[t_0,t_f]\rightarrow \mathbb{R}^{n\times n}$ as
\begin{eqnarray}
S(t) = \frac{\partial  \p}{\partial \q}(t,\bar{\q})\left[\frac{\partial \x}{\partial \q}(t,\bar{\q})\right]^{-1}.
\label{EQ:S}
\end{eqnarray}
Set
\begin{eqnarray}
\Delta \x_i := \frac{\partial \x}{\partial \q}(t_i,\bar{\q})\left[\frac{\partial \x}{\partial \q}(t_0,\bar{\q})\right]^{-1}\Delta \x,
\end{eqnarray}
it is clear that $\Delta \x_i$ is the first-order term of the deviation $\x(t_i,\q_*)-\bar{\x}(t_i)$. Substituting Eq.~(\ref{EQ:Delta_q}) and Eq.~(\ref{EQ:variation_switching_time}) into Eq.~(\ref{EQ:Delta_ti}), one gets 
\begin{eqnarray}
\Delta t_i = &-&\left[\frac{\partial H_1(\bar{\x}(t_i),\bar{\p}(t_i))}{\partial \x^T} + \frac{\partial H_1(\bar{\x}(t_i),\bar{\p}(t_i))}{\partial \p^T} S(t_i) \right] \Delta \x_i /\dot{H}_1(\bar{\x}(t_i),\bar{\p}(t_i))\nonumber\\
& +&\frac{d t_i(\bar{\q})}{d\q} O_{\q}(\abs{\Delta \x}^2)   + O_{t_i}(\abs{\Delta \q}^2) .\nonumber
\end{eqnarray}
Let
\begin{eqnarray}
\delta t_i := -\left[\frac{\partial H_1(\bar{\x}(t_i),\bar{\p}(t_i))}{\partial \x^T} + \frac{\partial H_1(\bar{\x}(t_i),\bar{\p}(t_i))}{\partial \p^T} S(t_i) \right] \Delta \x_i /\dot{H}_1(\bar{\x}(t_i),\bar{\p}(t_i)),
\label{EQ:delta_ti}
\end{eqnarray}
be the first-order term of $\Delta t_i$. Then, if $\Delta \x$ is infinitesimal,  it suffices to use $t_i + \delta t_i$  as the neighboring optimal feedback on switching times. 

Note that there may exist some profiles of the switching function $H_1(\bar{\x}(\cdot),\bar{\p}(\cdot))$ as is shown by the solid line in Figure \ref{Fig:H1_time}. Then, a small perturbation may result in the change on the number of switching times, as is illustrated by the two dashed lines in Figure \ref{Fig:H1_time}.
\begin{figure}[!ht]
 \begin{center}
 \includegraphics[trim=1cm 2.5cm 1cm 0cm, clip = true, width=4.0in,angle=0]{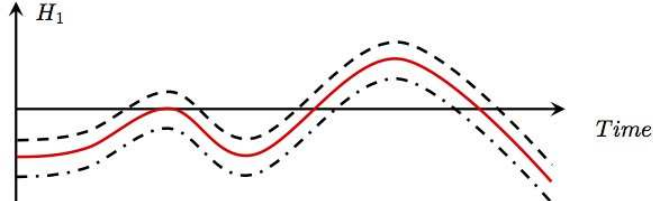}
 \end{center}
 \caption[]{Variations of switching times  with respect to initial perturbations.}
 \label{Fig:H1_time}
\end{figure}
However, Eq.~(\ref{EQ:delta_ti}) is unable to provide the feedback on switching times if the number of switching times on the neighboring extremal is different from that of the nominal one. Set $   [\Delta \p_r^T\ \Delta \p_v^T\ \Delta p_m]^T:=\Delta \p =\p(t_0,\q_*) - \p(t_0,\bar{\q})$ and $[\Delta \r^T\ \Delta \v^T\ \Delta m]^T := \Delta \x$.  According to Eq.~(\ref{EQ:Max_condition1}),  the switching function $H_1(\x_*,\bar{\p}(t_0) + \Delta \p)$ gives a natural feedback on the optimal thrust magnitude of the new extremal $\gamma(\cdot,\q_*)$ at $t_0$, i.e.,
\begin{eqnarray}
\rho&=& \Big[1+ \mathrm{sgn}\big(H_1(\bar{\x}(t_0)+\Delta \x,\bar{\p}(t_0)+\Delta \p)\big)\Big]/2 \nonumber\\
&=& \Big[1+ \mathrm{sgn}\Big(\frac{\|\bar{\p}_v(t_0) + \Delta \p_v\|}{\bar{m}(t_0) + \Delta m} u_{max}- \beta (\bar{p}_m(t_0) + \Delta p_m) u_{max} - 1\Big)\Big]/2,\nonumber
\end{eqnarray}
where $\mathrm{sgn}(\cdot)$  is the typical sign function. Thus, instead of using the first order term $\delta t_i$ in Eq.~(\ref{EQ:delta_ti}) to approximate switching times, one can directly check the sign of the switching function $H_1(\x_*,\bar{\p}(t_0)+\Delta \p)$ to generate the optimal thrust magnitude once $\Delta \p_v$ and $\Delta p_m$ are computed. The first-order Taylor expansion of $\p(t_0,\q_*)$ around $\bar{\q}$ is
\begin{eqnarray}
\Delta \p =\p(t_0,\q_*) - \p(t_0,\bar{\q}) = \frac{\partial \p}{\partial \q}(t_0,\bar{\q})\Delta \q^T + O_{\p}(\abs{ \Delta \q }^2),
\label{EQ:Delta_p0}
\end{eqnarray}
where $O_{\p}(\abs{ \Delta \q }^2)$ is the sum of second and higher order terms. 
  Substituting Eq.~(\ref{EQ:Delta_q}) into Eq.~(\ref{EQ:Delta_p0}) leads to
\begin{eqnarray}
 \Delta \p =  S(t_0)\Delta \x + \frac{\partial \p(t_0,\bar{\q})}{\partial \q} O_{\q}(\abs{ \Delta \x }^2) + O_{\p}(\abs{\Delta \q}^2).
 \label{EQ:Delta_p}
\end{eqnarray}
Denote by  $S_1\in\mathbb{R}^{3\times 7}$ the first three rows, $S_2\in\mathbb{R}^{3\times 7}$ the forth to sixth rows, and $S_3\in(\mathbb{R}^{7})^*$ the last row of the gain matrix $S$ such that 
\begin{eqnarray}
S = \left(\begin{array}{c}S_1\\
S_2\\
S_3
\end{array}\right).\nonumber
\end{eqnarray}
It is clear that $S_2(t_0)\Delta \x$ and $S_3(t_0)\Delta \x$  are the first order terms of $\Delta \p_v$ and $\Delta p_m$, respectively. Thus, if $\Delta \x$ is small enough, it is sufficient to use
\begin{eqnarray}
\rho &=& \Big[ 1+ sgn\Big(\frac{\|\bar{\p}_v(t_0) + S_2(t_0)\Delta \x \|}{\bar{m}(t_0) + \Delta m} u_{max} \nonumber\\
&-& \beta (\bar{p}_m(t_0) + S_3(t_0)\Delta \x ) u_{max} - 1\Big)\Big]/2,
\label{EQ:Thrust_magnitude}
\end{eqnarray}
as the neighboring optimal feedback on thrust magnitude. 

\subsubsection{Neighbouring optimal feedback on thrust direction}

According to Eq.~(\ref{EQ:Max_condition2}), if $\|\bar{\p}_v(t_0) + \Delta \p_v\|\neq 0$, the optimal thrust direction on the new extremal $\gamma(\cdot,\q_*)$ at $t_0$ is
\begin{eqnarray}
\boldsymbol{\tau}(\x_*,\bar{\p}(t_0)+\Delta \p) = \frac{\bar{\p}_v(t_0) + \Delta \p_v }{\|\bar{\p}_v(t_0) + \Delta \p_v\|}.\nonumber
\end{eqnarray}
 Analogously, assume $\Delta \x$ is infinitesimal, if $\|\bar{\p}_v(t_0) + S_2(t_0)\Delta \x \|\neq 0$, we can use
\begin{eqnarray}
\boldsymbol{\tau}(\x_*,\bar{\p}(t_0)+\delta \p(t_0)) = \frac{\bar{\p}_v(t_0) + S_2(t_0)\Delta \x }{\|\bar{\p}_v(t_0) + S_2(t_0)\Delta \x \|},
\label{EQ:Thrust_direction}
\end{eqnarray}
as the neighboring optimal feedback on the thrust direction.
\begin{remark}
One advantage of using Eq.~(\ref{EQ:Thrust_magnitude}) and Eq.~(\ref{EQ:Thrust_direction}) to generate neighboring optimal feedbacks is that  only  $4/7$ instead of the whole block of the time-varying gain matrix $S(\cdot)$ on $[0,t_f)$ is required to store in the onboard computer.
\label{RE:remark2}
\end{remark}


\section{Numerical implementations of the NOG}\label{SE:Implementation}

Once the perturbation $\Delta \x$ is measured  at  $t_0\in[0,t_f)$, it amounts to compute the two matrices $\frac{\partial \x}{\partial \q}(t_0,\bar{\q})$ and $\frac{\partial \p}{\partial \q}(t_0,\bar{\q})$ in order to compute the neighbouring optimal feedbacks in Eq.~(\ref{EQ:Thrust_magnitude}) and Eq.~(\ref{EQ:Thrust_direction}).

\subsection{Differential equations for $\frac{\partial \x}{\partial \q}(t,\bar{\q})$ and $\frac{\partial \p}{\partial \q}(t,\bar{\q})$}
 It follows from the classical results about solutions to ordinary differential equations that the trajectory $(\x(\cdot,\q),\p(\cdot,\q))$ and its time derivative $(\dot{\x}(\cdot,\q),\dot{\p}(\cdot,\q))$ on $[t_0,t_f]$ are continuously differentiable with respect to $\q$. Thus, taking derivative of Eq.~(\ref{EQ:canonical}) with respect to $\q$ on each subinterval $(t_i,t_{i+1})$ yields the homogeneous linear matrix differential equations
\begin{eqnarray}
\left[\begin{array}{c}
\frac{d}{dt}\frac{\partial \x}{\partial\q}(t,\bar{\q})\\
\frac{d}{dt}\frac{\partial \p}{\partial \q}(t,\bar{\q})
\end{array}
\right]
=
\left[
\begin{array}{cc}
H_{\p\x}(\bar{\x}(t),\bar{\p}(t)) &
H_{\p\p}(\bar{\x}(t),\bar{\p}(t))    \\
-H_{\x\x}(\bar{\x}(t),\bar{\p}(t))  &
-H_{\x\p}(\bar{\x}(t),\bar{\p}(t)) 
\end{array}
\right]
\left[\begin{array}{c}
\frac{\partial \x}{\partial\q}(t,\bar{\q})\\
\frac{\partial \p}{\partial \q}(t,\bar{\q})
\end{array}
\label{EQ:Homogeneous_matrix}
\right].
\end{eqnarray} 
Substituting the maximum condition in Eq.~(\ref{EQ:Max_condition2}) and the system dynamics in Eq.~(\ref{EQ:Sigma}) into the maximized Hamiltonian $H$, a direct derivation yeilds
\begin{eqnarray}
H_{\p\x} = H_{\x\p}^T = \left[\begin{array}{ccc}\boldsymbol{0}_3 & I_3& \boldsymbol{0}_{3\times 1}\\
-\mu\frac{I_3 \| \r \|^2 - 3 \r  \r^T}{\| \r \|^5} & \boldsymbol{0}_3 &- \rho \frac{\p_v}{\| \p_v\| m^2}\\
\boldsymbol{0}_{1\times 3} & \boldsymbol{0}_{1\times 3} & 0 \end{array}\right],\nonumber
\end{eqnarray}
\begin{eqnarray}
H_{\p\p} = \left[\begin{array}{ccc}\boldsymbol{0}_3 & \boldsymbol{0}_3& \boldsymbol{0}_{3\times 1}\\
\boldsymbol{0}_3 &\rho u_{max} \frac{I_3 \| \p_v\|^2 - \p_v \p_v^T}{m \| \p_v \|^3} &  \boldsymbol{0}_{3\times 1}\\
\boldsymbol{0}_{1\times 3} & \boldsymbol{0}_{1\times 3} & 0 \end{array}\right],\ \text{and}\nonumber
\end{eqnarray}
\begin{eqnarray}
H_{\x\x} = \left[\begin{array}{ccc}
-3\mu \frac{3(\p_v \r^T)\| \r\|^2 - 5 (\p_v^T\r)(\r\r^T)}{\| \r \|^7}& \boldsymbol{0}_3& \boldsymbol{0}_{3\times 1}\\
\boldsymbol{0}_3 &\boldsymbol{0}_3 &  \boldsymbol{0}_{3\times 1}\\
\boldsymbol{0}_{1\times 3} & \boldsymbol{0}_{1\times 3} & 2\rho u_{max}\frac{\| \p_v \|}{m^3} \end{array}\right],\nonumber
\end{eqnarray}
where $\boldsymbol{0}_i$ and $I_i$ denote the zero and the identity matrices of $\mathbb{R}^{i\times i}$, respectilvey, and $\boldsymbol{0}_{i\times j}$ denotes the zero matrix of $\mathbb{R}^{i\times j}$.  The two matrices $\frac{\partial \x}{\partial \q}(t,\bar{\q})$ and $\frac{\partial \boldsymbol{p}}{\partial \q}(t,\bar{\q})$ are discontinuous at  each switching time $t_i$ ($i=1,2,\cdots,k$). By virtue of Lemma 2.6 in Ref.~\cite{Noble:02}, the updating formulas for the two matrices at each switching time $t_i$ are 
\begin{eqnarray}
\begin{cases}
\frac{\partial \x}{\partial \q}(t_i^+,\bar{\q}) = \frac{\partial \x}{\partial \q}(t_i^-,\bar{\q})  - \Delta \rho_i \f_1(\bar{\x}(t_i),\bar{\boldsymbol{\tau}}(t_i))\frac{\partial t_i(\bar{\q})}{\partial \q},\\
\frac{\partial \p}{\partial \q}(t_i^+,\bar{\q}) = \frac{\partial \p}{\partial \q}(t_i^-,\bar{\q}) + \Delta \rho_i \frac{\partial \f_1}{\partial \x}(\bar{\x}(t_i),\bar{\boldsymbol{\tau}}(t_i))\bar{\p}(t_i)\frac{\partial t_i(\bar{\q})}{\partial \q},
\end{cases}
\label{EQ:update_formula}
\end{eqnarray}
where $\Delta \rho_i = \bar{\rho}(t_i^+) - \bar{\rho}(t_i^-)$,
$$\frac{\partial \f_1}{\partial \x} = \left[\begin{array}{ccc}
\boldsymbol{0}_3 & \boldsymbol{0}_3 & \boldsymbol{0}_{3\times 1}\\
\boldsymbol{0}_3 & \boldsymbol{0}_3 & -u_{max}\frac{\p_v}{\| \p_v\| m^2}\\
\boldsymbol{0}_{1\times3} & \boldsymbol{0}_{1\times3} &0
\end{array}\right],$$
and $d t_i(\bar{\q})/d\q$ can be computed by using Eq.~(\ref{EQ:variation_switching_time}).

Once the initial values $\frac{\partial \x}{\partial \q}(t_f,\bar{\q})$ and $\frac{\partial \p}{\partial \q}(t_f,\bar{\q})$ are given,   the two matrices $\frac{\partial \x}{\partial \q}(t,\bar{\q})$ and $\frac{\partial \boldsymbol{p}}{\partial \q}(t,\bar{\q})$ for $t\in[t_0,t_f]$ can be computed by integrating the differential equations in Eq.~(\ref{EQ:Homogeneous_matrix}) between switching times and by using the updating formulas in Eq.~(\ref{EQ:update_formula}) at each switching time.

\subsection{Computation of $\frac{\partial \x}{\partial \q}(t_f,\bar{\q})$ and $\frac{\partial \p}{\partial \q}(t_f,\bar{\q})$}

 Typically, the sweep variables are used to compute the initial values $\frac{\partial \x}{\partial \q}(t_f,\bar{\q})$ and $\frac{\partial \p}{\partial \q}(t_f,\bar{\q})$~\cite{Bryson:69,Hull:03,Schattler:12}. Note that the matrix
 \begin{eqnarray}
\left[\frac{\partial \x}{\partial \q}(t_f,\bar{\q}),
\frac{\partial \p}{\partial \q}(t_f,\bar{\q})\right] = \frac{d\F^{-1}(\bar{\q})}{d\q}\nonumber
\end{eqnarray}
is a set of basis vectors of the tangent space $T_{\bar{\z}_f}\mathcal{N}$ at $\bar{\z}_f=(\bar{\x}(t_f),\bar{\p}(t_f))$. Thus, to compute the initial values $\frac{\partial \x}{\partial \q}(t_f,\bar{\q})$ and $\frac{\partial \p}{\partial \q}(t_f,\bar{\q})$, it amounts to compute a basis of the tangent space $T_{\bar{\z}_f}\mathcal{N}$ at $\bar{\z}_f=(\bar{\x}(t_f),\bar{\p}(t_f))$.

\subsubsection{Initial values for the case of $l=n$}
If $l=n$, the final state is fixed since the submanifold $\mathcal{M}$ reduces to a singleton. Thus, in the case of $l=n$, one can simply set $\q = \p^T(t_f)$, which indicates
\begin{eqnarray}
\frac{\partial \x}{\partial \q}(t_f,\bar{\q}) = 0_n,\ \frac{\partial \p}{\partial \q}(t_f,\bar{\q}) = I_n.
\end{eqnarray}

\subsubsection{Initial values for the case of $0<l<n$}

Note that $\Pi(\mathcal{N}) \subset\Pi(\mathcal{L}_f)$ and $\Pi(\mathcal{L}_f) = \mathcal{M}$.  Thus, the subset $\Pi(\mathcal{N})$ is diffeomorphic to $(\mathbb{R}^{n-l})^*$ if the subset $\mathcal{N}$ is small enough. In analogy with parameterizing neighbouring extremals, if the subset $\mathcal{N}$ is small enough and $l<n$, there exists an invertible function $\boldsymbol{F}_1:\Pi(\mathcal{N})\rightarrow (\mathbb{R}^{n-l})^*$ such that both the function and its inverse $\F_1^{-1}$ are smooth. Then, for every $\x\in\Pi(\mathcal{N})$, there exists one and only one $\q_1  = (\mathbb{R}^{n-l})^*$ such that $\q_1 = \F_1(\x)$. According to the transversality condition in Eq.~(\ref{EQ:Transversality}), for every $(\x,\p)\in\mathcal{N}$, there exists a $\boldsymbol{\nu}\in(\mathbb{R}^l)^*$ such that 
\begin{eqnarray}
\p^T = \boldsymbol{\nu}{\nabla \phi(\x)}.
\end{eqnarray}
Let us define a function $\F_2:\mathcal{N}\rightarrow (\mathbb{R}^l)^*,\ (\x,\p)\mapsto\F_2(\x,\p)$ as
\begin{eqnarray}
\F_2(\x,\p) = \p^T \nabla\phi^T(\x) \left[\nabla \phi(\x)\nabla \phi^T(\x)\right]^{-1},\nonumber
\end{eqnarray}
such that $\boldsymbol{\nu} = \F_2(\x,\p)$. By Assumption \ref{AS:full_rank}, if the subset $\mathcal{N}$ is small enough, the function $\F_2$ is a diffeomorphism from the domain $\mathcal{N}$ onto its image. Thus, it is enough to set $\F=[\F_1,\F_2]$ such that $\q = [\q_1,\boldsymbol{\nu}]$.  Let $\bar{\q}_1 = \F_1(\bar{\x}(t_f))$, we have $\bar{\q}=(\bar{\q}_1,\bar{\boldsymbol{\nu}})$ where 
$
\bar{\boldsymbol{\nu}} =\F_2(\bar{\x}(t_f),\bar{\p}(t_f))$
  denotes the vector of the Lagrangian multipliers for the nominal extremal $\gamma(\cdot,\bar{\q})$ on $[0,t_f]$. A direct calculation leads to
\begin{eqnarray}
\frac{\partial \x}{\partial \q}(t_f,\bar{\q}) &=& \left[\frac{\partial \x}{\partial \q_1}(t_f,\bar{\q}),\ \frac{\partial \x}{\partial \boldsymbol{\nu}}(t_f,\bar{\q})\right],\label{EQ:initial_matrix3x}\\
\frac{\partial \p}{\partial \q}(t_f,\bar{\q}) &=& \left[\frac{\partial \p(t_f,\bar{\q})}{\partial \q_1},\ \frac{\partial \p(t_f,\bar{\q})}{\partial \boldsymbol{\nu}}\right]\nonumber\\
& =&\left[\sum_{i=1}^{l}{\bar{{\nu}}_i  \nabla^2 \phi_i}({\x}(t_f,\bar{\q}))\frac{\partial  \x(t_f,\bar{\q})}{\partial \q_1},\ {\nabla \phi^T(\x(t_f,\bar{\q}))} \right],
\label{EQ:initial_matrix3}
\end{eqnarray}
where $\phi_i:\mathcal{X}\rightarrow \mathbb{R}$ and $\bar{\nu}_i\in\mathbb{R}$ for $i=1,2,\cdots,l$ are the elements of the vector-valued function $\phi(\x)$ and the vector $\bar{\boldsymbol{\nu}}$, respectively.   Since $\x(t_f,\q) \in\Pi(\mathcal{N})$ is not a function of $\boldsymbol{\nu}$, there holds
\begin{eqnarray}
\frac{\partial \x(t_f,\bar{\q})}{\partial \boldsymbol{\nu}} = \boldsymbol{0}_{n\times l}.
\label{EQ:initial_matrix2}
\end{eqnarray}
Note that, except the matrix $\frac{\partial \x}{\partial \q_1}(t_f,\bar{\q})$, all the quantities for computing the initial conditions in Eq.~(\ref{EQ:initial_matrix3x}) and Eq.~(\ref{EQ:initial_matrix3}) are available.
Let us
take the differentiation of $\phi(\x(t_f,\q)) = 0$ with respect to $\q_1$, we get
\begin{eqnarray}
\nabla \phi(\x(t_f,\bar{\q}))\frac{\partial \x(t_f,\bar{\q}) }{\partial \q_1}= 0.
\label{EQ:initial_matrix1}
\end{eqnarray}
Note that all the column vectors of the matrix $\frac{\partial \x(t_f,\bar{\q}) }{\partial \q_1}$ constitute a basis of the tangent space $T_{\bar{\x}(t_f)}\mathcal{M}$.
Once the matrix ${\nabla \phi(\x(t_f,\bar{\q}))}$ is given, one can compute the full-rank matrix $\frac{\partial \x}{\partial \q_1}(t_f,\bar{\q})$ by a Gram-Schmidt orthogonalization, which can be numerically done by employing the {\it gram} function of MATLAB.

Up to now, all the quantities for computing the initial conditions $\frac{\partial \x}{\partial \q}(t_f,\bar{\q})$ in Eq.~(\ref{EQ:initial_matrix3x}) and $\frac{\partial \p}{\partial \q}(t_f,\bar{\q})$ in Eq.~(\ref{EQ:initial_matrix3}) are available  for $l<n$.

\subsection{Riccati differential equation}

 Note that one has to solve a $2\times n^2$ order of differential equations in order to compute the matrix $S(t)$ if using Eq.~(\ref{EQ:Homogeneous_matrix}) and Eq.~(\ref{EQ:update_formula}).
 In this subsection,  the differential equations of the gain matrix $S(t)$ will be derived such that only $n^2$ order of differential equations are required to solve.
 
According to Eq.~(\ref{EQ:S}), we have  
$$S(\cdot)\frac{\partial \x}{\partial \q}(\cdot,\bar{\q}) = \frac{\partial \p}{\partial \q}(\cdot,\bar{\q}),$$
on $[0,t_f]$. Differentiating this equation with respect to time yields
\begin{eqnarray}
\dot{S}(\cdot)\frac{\partial \x}{\partial \q}(\cdot,\bar{\q}) + {S}(\cdot)\frac{\partial \dot{\x}}{\partial \q}(\cdot,\bar{\q}) = \frac{\partial \dot{\p}}{\partial \q}(\cdot,\bar{\q}),\nonumber
\end{eqnarray}
on $[0,t_f]$. Substituting Eq.~(\ref{EQ:Homogeneous_matrix}) into this equation, we hence obtain
\begin{eqnarray}
\dot{S}(\cdot)  &=& -H_{\x\x}(\bar{\x}(\cdot),\bar{\p}(\cdot)) - H_{\x\p}(\bar{\x}(\cdot),\bar{\p}(\cdot)) S(\cdot) \nonumber\\
&-& S(\cdot) H_{\p\x}(\bar{\x}(\cdot),\bar{\p}(\cdot))  - S(\cdot) H_{\p\p}(\bar{\x}(\cdot),\bar{\p}(\cdot)) S(\cdot),
\label{EQ:riccati_equation}
\end{eqnarray}
on $[0,t_f]$, which is exactly the Riccati-type differential equation in Refs.~\cite{Breakwell:63,Speyer:68,Bryson:69}. According to Eq.~(\ref{EQ:update_formula}), the gain matrix $S(\cdot)$ is discontinuous at each switching time $t_i$. Assume the matrix $\frac{\partial \x}{\partial \q}(t_i^-,\bar{\q})$ is nonsingular, multiplying $ \frac{\partial \p}{\partial \q}(t_i^-,\bar{\q})$ by  $\left[\frac{\partial \x}{\partial \q}(t_i^-,\bar{\q})\right]^{-1}$ and taking into account  Eq.~(\ref{EQ:update_formula}), one obtains
\begin{eqnarray}
S(t_i^- ) &=& \frac{\partial \p}{\partial \q}(t_i^-,\bar{\q})\left[ \frac{\partial \x}{\partial \q}(t_i^-,\bar{\q})\right]^{-1}\nonumber\\
&=& \left[ \frac{\partial \p}{\partial \q}(t_i^+,\bar{\q}) - \Delta \rho_i \frac{\partial  \f_1}{\partial \x}(\bar{\x}(t_i),\bar{\boldsymbol{\tau}}(t_i))\bar{\p}(t_i)\frac{\partial t_i(\bar{\q})}{\partial \q}\right]\nonumber\\
&\times&\left[ \frac{\partial \x}{\partial \q}(t_i^+,\bar{\q})  + \Delta \rho_i \f_1(\bar{\x}(t_i),\bar{\boldsymbol{\tau}}(t_i))\frac{\partial t_i(\bar{\q})}{\partial \q}\right]^{-1}\nonumber\\
&=& \left\{ S(t_i^+) - \Delta \rho_i \frac{\partial  \f_1}{\partial \x}(\bar{\x}(t_i),\bar{\boldsymbol{\tau}}(t_i))\bar{\p}(t_i)\frac{\partial t_i(\bar{\q})}{\partial \q}\left[\frac{\partial \x}{\partial \q}(t_i^+,\bar{\q})\right]^{-1}\right\}\nonumber\\
&\times&\frac{\partial \x}{\partial \q}(t_i^+,\bar{\q})\left[ \frac{\partial \x}{\partial \q}(t_i^+,\bar{\q})  + \Delta \rho_i \f_1(\bar{\x}(t_i),\bar{\boldsymbol{\tau}}(t_i))\frac{\partial t_i(\bar{\q})}{\partial \q}\right]^{-1}.
\label{EQ:S_long}
\end{eqnarray}
Let us define a vector-valued function $R(t_i):  \mathbb{R}_+\rightarrow (\mathbb{R}^n)^* $ as
\begin{eqnarray}
R(t_i)  = \frac{d t_i(\bar{\q})}{d\q} \left[\frac{\partial \x}{\partial \q}(t_i^+,\bar{\q})\right]^{-1}.\nonumber
\end{eqnarray}
Substituting this equation into Eq.~(\ref{EQ:variation_switching_time}) yields
\begin{eqnarray}
R(t_i) = - \left[ \frac{\partial H_1(\bar{\x}(t_i),\bar{\p}(t_i))}{\partial \x^T} + \frac{\partial H_1(\bar{\x}(t_i),\bar{\p}(t_i))}{\partial \p^T}S(t_i^+)\right]/\dot{H}_1(\bar{\x}(t_i),\bar{\p}(t_i)).\nonumber
\end{eqnarray}
Given a nonsingular matrix $\boldsymbol{A}\in\mathbb{R}^{n\times n}$ and two vectors $\boldsymbol{b}\in\mathbb{R}^n$ and $\boldsymbol{c}\in\mathbb{R}^n$, if the matrix $\boldsymbol{A} + \boldsymbol{b}\boldsymbol{c}^T$ is nonsingular, the equation
\begin{eqnarray}
\left(\boldsymbol{A} + \boldsymbol{b}\boldsymbol{c}^T\right)^{-1} = \boldsymbol{A}^{-1} - \frac{\boldsymbol{A}^{-1}\boldsymbol{b}\boldsymbol{c}^T\boldsymbol{A}^{-1}}{1 + \boldsymbol{c}^T\boldsymbol{A}^{-1}\boldsymbol{b}},
\label{EQ:lemma_equation}
\end{eqnarray}
is satisfied (cf. Lemma 6.1.4 in Ref.~\cite{Schattler:12}).
Thus, if the matrix $\frac{\partial \x}{\partial \q}(t_i^+,\bar{\q})$ is nonsingular, taking into account Eq.~(\ref{EQ:lemma_equation}), one gets
\begin{eqnarray}
&&\frac{\partial \x}{\partial \q}(t_i^+,\bar{\q})\left[ \frac{\partial \x}{\partial \q}(t_i^+,\bar{\q})  + \Delta \rho_i \f_1(\bar{\x}(t_i),\bar{\boldsymbol{\tau}}(t_i))\frac{d t_i(\bar{\q})}{d \q}\right]^{-1}\nonumber\\
 &=&I_n - \Delta \rho_i \frac{ \f_1(\bar{\x}(t_i),\bar{\boldsymbol{\tau}}(t_i)) R(t_i) }{1 + \Delta \rho_i R(t_i)\f_1(\bar{\x}(t_i),\bar{\boldsymbol{\tau}}(t_i)) }.\nonumber
\end{eqnarray}
Substituting this equation into Eq.~(\ref{EQ:S_long}), we eventually obtain the result
\begin{eqnarray}
S(t_i^-) &=& \left[S(t_i^+) -\Delta \rho_i \frac{\partial \f_1}{\partial {\x}}(\bar{\x}(t_i),\bar{\boldsymbol{\tau}}(t_i))\bar{\p}(t_i) R(t_i)\right]\nonumber\\
&&\times \left[I_n -\Delta \rho_i \frac{\f_1(\bar{\x}(t_i),\bar{\boldsymbol{\tau}}(t_i)) R(t_i)}{1 + \Delta \rho_i R(t_i)\f_1(\bar{\x}(t_i),\bar{\boldsymbol{\tau}}(t_i))}\right].
\label{EQ:update_riccati}
\end{eqnarray}
This formula provides the required initial condition for Eq.~(\ref{EQ:riccati_equation}) on the interval $(t_{i-1},t_i)$. Then, the gain matrix  $S(\cdot)$ can be propagated further backward by integrating the Ricatti differential equation in Eq.~(\ref{EQ:riccati_equation}).  Note that $S(t_f) = \infty$ since the matrix $\frac{\partial \x}{\partial \q}(t_f,\bar{\q})$ is singular. One can use Eq.~(\ref{EQ:Homogeneous_matrix}) to integrate backward from $t_f$ on a short interval $[t_s,t_f]$ with $t_s < t_f$ to get the matrices $\frac{\partial \x}{\partial \q}(t_s,\bar{\q})$ and $\frac{\partial \p}{\partial \q}(t_s,\bar{\q})$. Then, substituting the matrices $\frac{\partial \x}{\partial \q}(t_s,\bar{\q})$ and $\frac{\partial \p}{\partial \q}(t_s,\bar{\q})$ into the matrix $S(t_s)$, one can use Eq.~(\ref{EQ:riccati_equation}) and Eq.~(\ref{EQ:update_riccati}) to get $S(\cdot)$ on $[0,t_s]$. Once the matrix $S(\cdot)$ on $[0,t_f)$ is computed offline, the matrices $S_2(\cdot)$ and $S_3(\cdot)$ on $[0,t_f)$ can be stored in the onboard computer such that the online computation is just to solve Eq.~(\ref{EQ:Thrust_magnitude}) and Eq.~(\ref{EQ:Thrust_direction}).

If the sweep variables (cf. Chapter 6 in Ref.~\cite{Bryson:69}, Chapter 5 in Ref.~\cite{Schattler:12}, or Chapter 11 in Ref.~\cite{Hull:03})  are employed to calculate the initial values $\frac{\partial \x}{\partial \q}(t_f,\bar{\q})$ and $\frac{\partial \p}{\partial \q}(t_f,\bar{\q})$, to compute the neighboring optimal feedbacks in Eq.~(\ref{EQ:Thrust_magnitude}) and Eq.~(\ref{EQ:Thrust_direction}), not only the gain matrix $S(t)$ but also two other time-varying matrices of $\mathbb{R}^{n\times l}$ and $\mathbb{R}^{l\times l}$ have to be computed offline. Thus, the method of this paper not only demands less storage capacity (cf. Remark \ref{RE:remark2}) but also requires less offline computational time. 


\section{Numerical Example}\label{SE:Numerical}

In this section, we consider to control a spacecraft from an inclined  elliptic orbit to the Earth geostationary orbit. Denote by $a$, $e$, $i$, $\omega$, $\Omega$, $f$ the  semi-major axis, the eccentricity, the inclination, the argument of periapsis, the argument of ascending node, and the true anomaly of the classical orbital elements (COE). The conditions for initial and final orbits are presented in Table \ref{Tab:Parameters} in terms of the COE.
\begin{table}[!ht]
    \renewcommand{\arraystretch}{1}
    \centering
     \caption[]{The initial and final conditions in terms of the COE.}
    \begin{tabular}{ccc}
    \hline\hline
 {COE}  & {Initial conditions}  &{Final conditions}\\
    \hline 

     $a$     &  26,571.429 km                & 42,165.000\ km        \\
     $e$     & 0.750                & 0     \\
    $i$     &  30.000 deg         & 0      \\
     $\omega$     & 0         & Undefined   \\
     $\Omega$     & 0           & Undefined     \\
    $f$  & $\pi$    & Undefined  \\ \hline
    \end{tabular}
\label{Tab:Parameters}
\end{table}
The Earth gravitational constant $\mu$ in Eq.~(\ref{EQ:system}) equals $398600.47$ km$^{3}$$s^{-2}$. The maximum thrust of the engine is $2.0$ N and the specific  impulse of the engine is $I_{sp} = 2000.0$ s. Let $g_0 = 9.8$ m/s$^2$ be the standard gravity at the surface of the Earth, we have $\beta = 1/(I_{sp}g_0) = 5.1\times 10^{-5}$ m$^{-2}$.  The initial mass $m_0$ of the spacecraft is $300.0$ kg. We specify the final time as $t_f = 157.88$ hours.  

 In order to achieve a stable numerical computation \cite{Caillau:03}, we use the {\it modified elementary orbital elements} (MEOE), 
\begin{eqnarray}
P &=& a(1 - e^2),\nonumber\\
e_x &=& e\cos(\omega + \Omega),\nonumber\\
e_y &=& e\sin(\omega+\Omega),\nonumber\\
h_x &=& \tan(i/2)\cos(\Omega),\nonumber\\
h_y &=& \tan(i/2)\sin(\Omega),\nonumber\\
l&=& f+\omega+\Omega,\nonumber
\end{eqnarray}
to compute optimal trajectories.  Note that the initial true longitude is $l_0 = \pi$, see Table \ref{Tab:Parameters}. In order to realize a multi-burn trajectory, we specify the final true longitude as $l_f = 9\times2\pi$ such that the spacecraft flies  9 revolutions around the Earth to get to the final orbit.

\subsection{Trajectory computation}

One can combine the final boundary condition in Table \ref{Tab:Parameters} and the transversality condition in Eq.~(\ref{EQ:Transversality}) to formulate a TPBVP~\cite{Pan:13}.  Then, it is enough to find the zero of this TPBVP in order to get the optimal solution.  A simple shooting method is not stable to solve the TPBVP because one usually does not know a priori the structure of the optimal control function. Thus, we use a regularization procedure developed in Ref.~\cite{Gergaud:06} to first get an energy-optimal trajectory with the same boundary conditions. Then, a homotopy method is employed to get the low-thrust fuel-optimal trajectory with a bang-bang control. The 3-dimensional position vector $\r(\cdot)$ on $[0,t_f]$ is plotted in Figure \ref{Fig:Transferring_Orbit3},
\begin{figure}[!ht]
 \begin{center}
 \includegraphics[trim=0cm 0cm 0cm 0cm, clip=true, width=0.8\textwidth, angle=0]{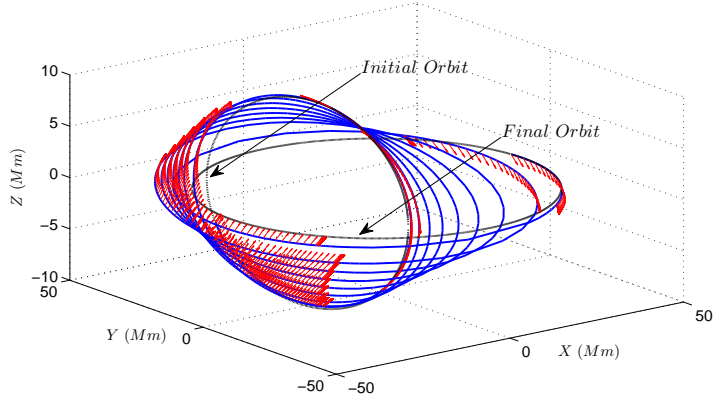}
 \end{center}
 \caption[]{The 3-dimensional trajectory $\r(\cdot)$ on $[0,t_f]$ for the low-thrust multi-burn fuel-optimal orbital transfer problem in a Cartesian coordinate system. The red arrows denote the thrust direction on burn arcs.}
 \label{Fig:Transferring_Orbit3}
\end{figure}
which shows that all the burn arcs occur around the apogees and perigees. To see the regularity conditions, Figure \ref{Fig:Transferring_Orbit4} plots the profiles of $\rho(\cdot)$, $H_1(\cdot)$, and $\| \p_v(\cdot)\|$ with respect to time on $[0,t_f]$. It is seen from this figure that the number of burn arcs along the low-thrust fuel-optimal trajectory is 13 with 24 switching points and that each switching point is regular, i.e., Assumption \ref{AS:regular_switching} holds along the computed extremal.
\begin{figure}[!ht]%
 \begin{center}
 \includegraphics[trim=0cm 0.0cm 0cm 0cm, clip=true,  width=0.9\textwidth, angle=0]{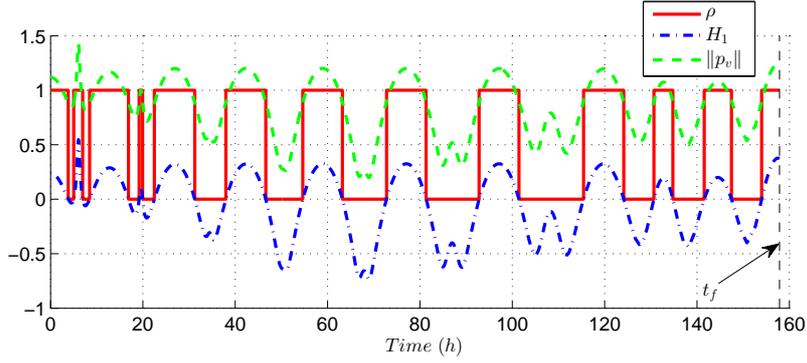}
 \end{center}
 \caption[]{The profiles of $\rho(\cdot)$, $H_1(\cdot)$, and $\parallel \p_v(\cdot)\parallel$ with respect to time on $[0,t_f]$ along the low-thrust multi-burn fuel-optimal trajectory.}
 \label{Fig:Transferring_Orbit4}
\end{figure}
The profiles of semi-major axis $a$, eccentricity $e$, and inclination $i$ along the low-thrust fuel-optimal trajectory are plotted in Figure \ref{Fig:a_e_i}.
\begin{figure}[!ht]
 \begin{center}
 \includegraphics[trim=0cm 0.0cm 0cm 0cm, clip=true, width=0.9 \textwidth, angle=0]{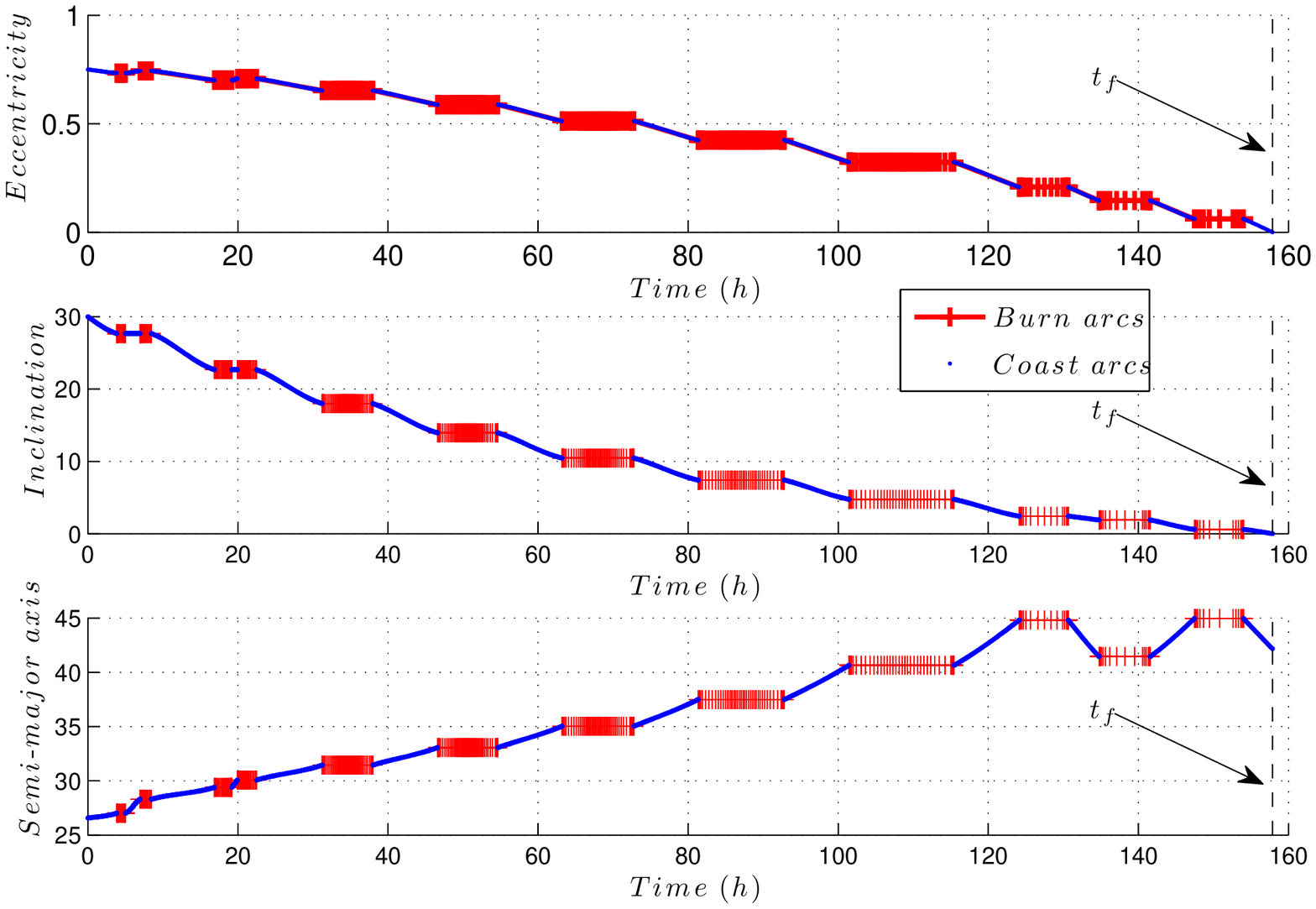}
 \end{center}
 \caption[]{The profiles of eccentricity $e$, inclination $i$, and semi-major axis $a$ against time  along the low-thrust multi-burn fuel-optimal trajectory.}
 \label{Fig:a_e_i}
\end{figure}

\subsection{Existence conditions and focal points}

Note that, except the final mass $m_f$, all other final states are fixed if we use the {\it MEOE} as states such that $$\x = (P,e_x,e_y,h_x,h_y,l,m)^T$$ and 
$$\p=(p_{p},p_{e_x},p_{e_y},p_{h_x},p_{h_y},p_l,p_m)^T.$$
 Thus, applying Eqs.~(\ref{EQ:initial_matrix2}--\ref{EQ:initial_matrix3}), we get the initial condition as
\begin{eqnarray}
\frac{\partial \x}{\partial \q}(t_f,\bar{\q}) &=& \left(\begin{array}{cc}\boldsymbol{0}_6 & \boldsymbol{0}_{6\times1}\\
\boldsymbol{0}_{1\times 6} & 1\end{array}\right),\nonumber\\
 \frac{\partial \p}{\partial \q}(t_f,\bar{\q}) &=&  \left(\begin{array}{cc}I_6 & \boldsymbol{0}_{6\times1}\\
\boldsymbol{0}_{1\times 6} & 0\end{array}\right).\nonumber
\end{eqnarray}
Then, starting from this initial condition, we propogate Eq.~(\ref{EQ:Homogeneous_matrix}) backward from the final time $t_f$ and use the updating formulas in Eq.~(\ref{EQ:update_formula}) at each switching time to compute the matrices ${\partial \x(\cdot,\bar{\q})}/{\partial \q}$ and ${\partial \p(\cdot,\bar{\q})}/{\partial \q}$ on $[0,t_f]$.

For notational simplicity, let $\delta (\cdot) = \det\left[\frac{\partial \x}{\partial \q}(\cdot,\bar{\q})\right]$ on $[0,t_f]$. To have a clear view, the profile of $\text{sgn}(\delta(\cdot))\times |\delta(\cdot)|^{1/20}$ instead of $\delta(\cdot)$ on $[0,t_f]$ is ploted in the top subplot of Figure \ref{Fig:det_4xtf}. Note that $\text{sgn}(\delta(\cdot))\times |\delta|^{1/20}$ on $[0,t_f]$ can capture the sign property of $\delta(\cdot)$ on $[0,t_f]$. We can clearly see from this figure that Conditions \ref{CON:disconjugacy_bang} and \ref{AS:Transversal} are satisfied on $[0,t_f)$. According to Theorem \ref{TH:optimality}, the low-thrust multi-burn fuel-optimal trajectory on $[0,t_f]$ realizes a local optimum. In addition, according to  Corollary \ref{CO:existence}, for every sufficiently small deviation $\Delta \x$ from the nominal trajectory $\bar{\x}(\cdot)$ at every time $t_0\in[0,t_f)$, there exists a neighbouring extremal $(\x(\cdot),\p(\cdot))$ on $[0,t_f]$ in $\mathcal{F}$ such that $\x(t_0) = \bar{\x}(t_0) + \Delta \x$. Thus, the NOG  can be constructed along the computed extremal trajectory.
\begin{figure}[!ht]
 \begin{center}
 \includegraphics[trim=0cm 0.0cm 0cm 0cm, clip=true, width=0.9\textwidth, angle=0]{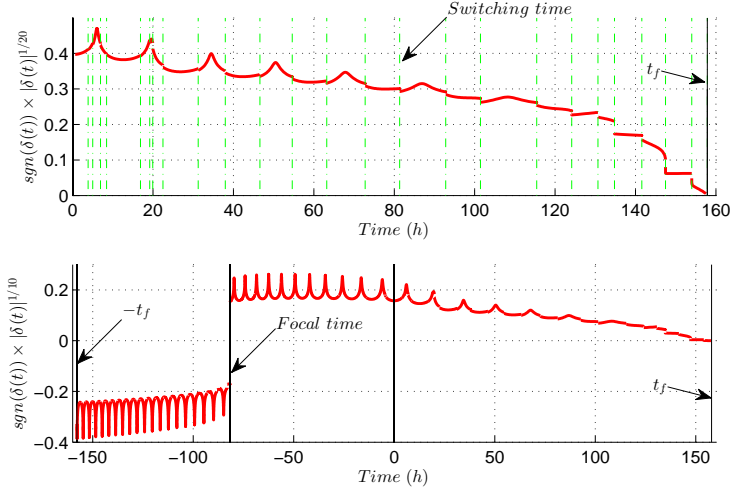}
 \end{center}
 \caption[]{The top plot is the profile of $\text{sgn}(\delta(t))\times |\delta(t)|^{1/20}$ on $[0,t_f]$ and the bottom plot is the profile of $\text{sgn}(\delta(t))\times |\delta(t)|^{1/10}$ on $[-t_f,t_f]$ .}
 \label{Fig:det_4xtf}
\end{figure}

In order to see the occurrence of focal points or to see the sign changes of $\delta(t)$, the profile of $\text{sgn}(\delta(\cdot))\times |\delta(\cdot)|^{1/10}$ on the extended time interval $[-t_f,t_f]$  is plotted in the bottom subplot of Figure \ref{Fig:det_4xtf}. Note that there exists a sign change of $\delta(t)$ at the switching time $t_c\approx -81.716 $ h. Thus,  Condition \ref{AS:Transversal} is violated at $t_c $, i.e., a focal point occurs at $t_c$, which implies that the nominal extremal $(\bar{\x}(\cdot),\bar{\p}(\cdot))$ on $[t_0,t_f]$ is not optimal any more if $t_0<t_c$ \cite{Caillau:15}. In addition,  as is shown by Corollary \ref{CO:existence}, no matter how small the absolute value $|\varepsilon|>0$ is, there exist some unit vectors $\boldsymbol{\eta}\in\mathbb{S}^{n-1}$  such that $\bar{\x}(t_c)+\varepsilon \boldsymbol{\eta} \neq \Pi(\gamma(t_c,\q))$ for every $\q\in\F(\mathcal{N})$. Hence, though the JC of Refs.~\cite{Chuang:96,Breakwell:63,Pontani:15:1,Pontani:15:2,Pontani:15:3,Bryson:69} is satisfied, it is impossible to construct the NOG along the computed extremal on $[t_0,t_f]$ with $t_0<t_c$ since none of neighboring extremals can pass through the point $\bar{\x}(t_c)+\varepsilon \boldsymbol{\eta}$ (cf. Proposition \ref{PR:Diffeomorphism1}).


\subsection{Tests of the NOG}

Let $t_0 =0$ and $\boldsymbol{\eta} = [1,1,1,1,1,1,1]^T$. A series of deviations $\Delta \x =  \varepsilon\boldsymbol{\eta}$ for $\varepsilon\in[0,1.0\times10^{-4}]$ are considered as the disturbances on the initial state $\x_0$. Then, assuming no further perturbations occur for $t>t_0$, for every $\varepsilon\in[0,1.0\times10^{-4}]$, the trajectories starting from the point $\x_0+\Delta \x$ associated with 
the neighbouring optimal feedback control  in Eq.~(\ref{EQ:Thrust_magnitude}) and Eq.~(\ref{EQ:Thrust_direction}) as well as the nominal  control are computed. Hereafter, we say the trajectories associated with the neighboring optimal feedback control as the neighboring optimal ones, and we say the trajectories associated with the nominal control as the perturbed ones.  The final values of $a$, $e_x$, and $e_y$ with respect to $\varepsilon\in[0,1.0\times10^{-4}]$ for the neighboring optimal trajectories and for the perturbed trajectories are plotted in Figure \ref{Fig:error_a_ex_ey}.
\begin{figure}[!ht]
 \begin{center}
 \includegraphics[trim=0cm 0.5cm 0cm 0cm, clip=true, width=0.8\textwidth, angle=0]{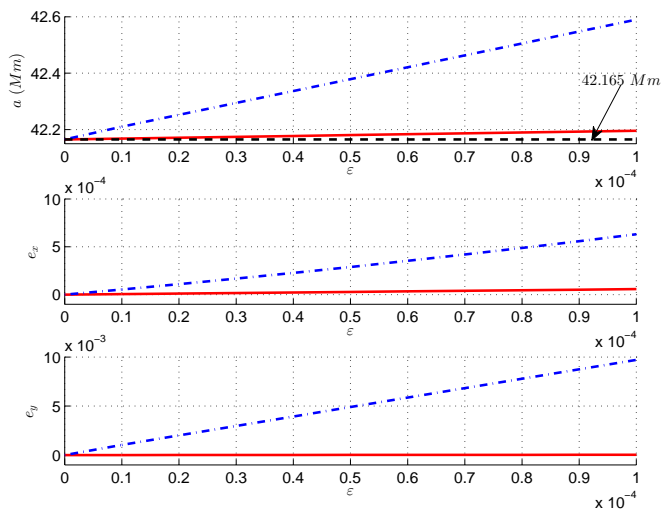}
 \end{center}
 \caption[]{Plots of the final $a$, $e_x$, and $e_y$ with respect to $\varepsilon$ on $[0,1\times 10^{-4}]$. The solid and dashed lines denote the final values for the neighboring optimal trajectories and for the perturbed trajectories, respectively.}
 \label{Fig:error_a_ex_ey}
\end{figure}
 Besides, the final values of $h_x$, $h_y$, and $l$ with respect to  $\varepsilon\in[0,1.0\times10^{-4}]$ for the neighboring optimal trajectories and the perturbed trajectory are plotted in Figure \ref{Fig:error_hx_hy_l}. 
 \begin{figure}[!ht]
 \begin{center}
 \includegraphics[trim=0cm 0.0cm 0cm 0cm, clip=true, width=0.8\textwidth, angle=0]{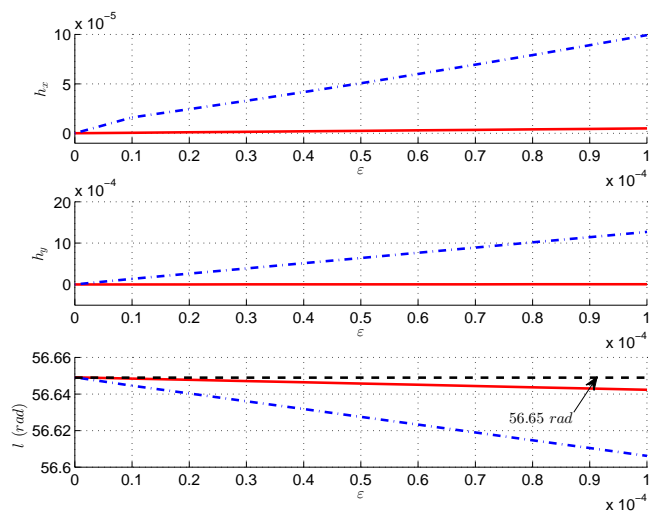}
 \end{center}
 \caption[]{Plots of the final $h_x$, $h_y$, and $l$ with respect to $\varepsilon$ on $[0,1\times 10^{-4}]$. The solid and dashed lines denote the final values for the neighboring optimal trajectories and for the perturbed trajectories.}
 \label{Fig:error_hx_hy_l}
\end{figure}
As is seen from Figure \ref{Fig:error_a_ex_ey}, when $\varepsilon$ increases up to $1.0\times 10^{-4}$, while the error of the final semi-major axis $a$ for the neighboring optimal trajectory remains small, that for the perturbed trajectory increases up to approximately $500.0$ km. We can also see from Figures  \ref{Fig:error_a_ex_ey} and \ref{Fig:error_hx_hy_l}  that the final values of $e_x$, $e_y$, $h_x$,  $h_y$, and $l$ for the neighboring optimal trajectories keep almost unchanged for $\varepsilon\in[0,1.0\times 10^{-4}]$. However, the final values of $e_x$, $e_y$, $h_x$,  $h_y$, and $l$ for the perturbed trajectories increase rapidly with the increase of $\varepsilon$ on $[0,1.0\times 10^{-4}]$. Therefore, the neighboring optimal feedback control in Eq.~(\ref{EQ:Thrust_magnitude}) and Eq.~(\ref{EQ:Thrust_direction}) greatly reduce the errors of final conditions.
 
 To see the advantage of using Eq.~(\ref{EQ:Thrust_magnitude}) rather than Eq.~(\ref{EQ:delta_ti}) to provide the neighboring optimal feedback on thrust magnitude, the profiles of switching function $H_1$ on the time interval $[0,25]$ along the neighboring extremals with respect to $\varepsilon \in[0,0.015]$ are plotted in Figure \ref{Fig:variaion_H1}. 
 \begin{figure}[!ht]
 \begin{center}
 \includegraphics[trim=-1cm 0.0cm 0cm 0cm, clip=true, width=0.8\textwidth]{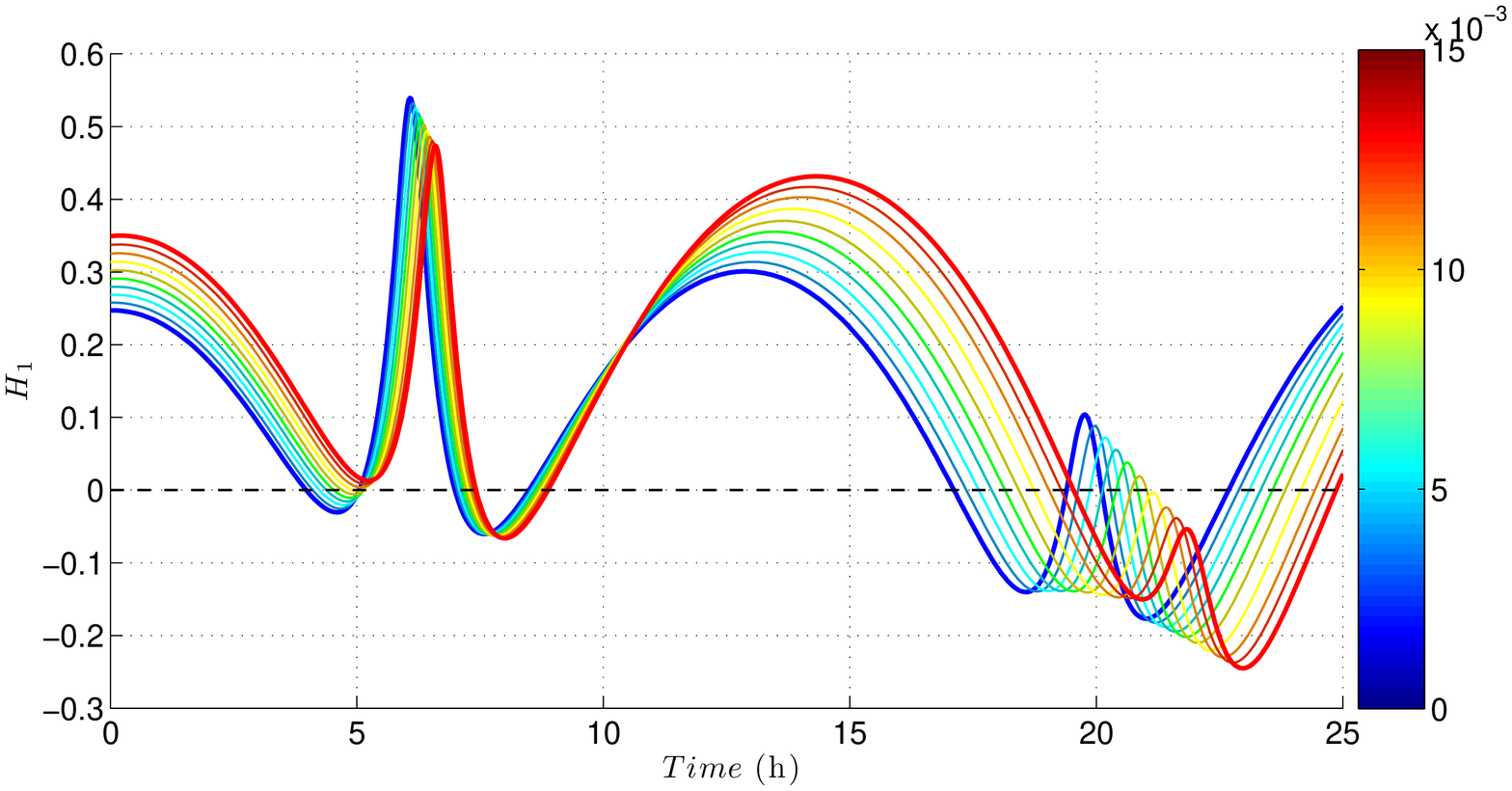}
 \end{center}
 \caption[]{The profiles of the switching functions $H_1$ on the time interval $[0,25]$ along the neighboring optimal trajectories with respect to $\varepsilon\in[0,1.5\times 10^{-2}]$. }
 \label{Fig:variaion_H1}
\end{figure}
We can clearly see from this figure that some switching times disappear around $t=5$ and $20$ with the increase of $\varepsilon$. In this case, while Eq.~(\ref{EQ:delta_ti}) cannot capture the variations of switching times, one can still compute the thrust magnitude of neighboring optimal trajectories by using Eq.~(\ref{EQ:Thrust_magnitude}). Though the disturbances are relativly big as $\varepsilon$ takes values up to $0.015$, Figure \ref{Fig:variaion_H1} shows the potential failure of using Eq.~(\ref{EQ:delta_ti}).

\section{Conclusion}\label{SE:Conclusion}

The neighbouring optimal feedback control strategy for fixed-time low-thrust multi-burn orbital transfer problems is established in this paper through constructing a parameterized family of  neighboring extremals around the nominal one. Two  conditions, including the JC and the TC,  sufficient for the existence of neighbouring extremals in an infinitesimal neighborhood of a bang-bang extremal are formulated. As a byproduct,  the sufficient conditions for the local optimality of bang-bang extremals are obtained.
Then, through deriving the first-order Taylor expansion of the paramterised neighboring extremals, the neighboring optimal feedbacks on thrust direction as well as on thrust magnitude are presented. The formulas of the neighboring optimal feedbacks show that to store only $4/7$ rather than the whole block of a gain matrix  of $\mathbb{R}^{n\times n}$ in the onboard computer is sufficient to realize the online computation.  Finally, a fixed-time low-thrust orbital transfer from an inclined elliptic orbit to the Earth geostationary orbit is computed, and various initial perturbations are tested to show that the NOG developed in this paper significantly reduces the errors of final conditions. The NOG for open-time multi-burn orbital transfers  will be studied in the subsequent research.












\begin{thebibliography}{}




\bibitem{Lu:94} Lu, P., ``A General Nonlinear Guidance Law,'' \textit{AIAA Paper 94-3632}, Aug. 1994.


\bibitem{Kelley:62} Kelley, H. J., ``{Guidance Theory and Extremal Fields},'' \textit{IRE Transation on Automatic Control}, Vol. AC-7, No. 5, 1962, pp. 75-82.

\bibitem{Kelley:64} Kelley, H. J.,  ``{An Optimal Guidance Approximation Theory},'' \textit{IEEE Trans.}, Vol. AC-9, 1964, pp. 375-380.

\bibitem{Lee:65} Lee, I., ``{Optimal Trajectory, Guidance, and Conjugate Points},'' \textit{Information and Control}, Vol. 8, 1965, pp.589--606.

\bibitem{Chuang:96} Chuang, C.-H., Goodson, T. D., Ledsinger, L. A., and Hanson, J.,  ``{Optimality and Guidance for Plannar Multiple-Burn Orbital Transfers},'' \textit{Journal of Guidance, Control, and Dynamics}, Vol. 23, No. 2, 1996, pp. 241-250.


\bibitem{Lu:08} Lu, P., Griffin, B., Dukeman, G., and Chavez, F.,  ``{Rapid Optimal Multi-Burn Ascent Planning and Guidance},'' \textit{Journal of Guidance, Control, and Dynamics}, Vol. 31, No. 6, 2008, pp. 1156--1164.


\bibitem{Lu:10} Lu, P., and Pan, B., ``{Highly Constrained Optimal Launch Ascent Guidance},''  \textit{Journal of Guidance, Control, and Dynamics}, Vol. 33, No. 3, 2010, pp. 756--767.


\bibitem{Lu:03} Lu, P., Sun, H., and Tsai, B., ``Closed-Loop Endoatmospheric Ascent Guidance,'' \textit{Journal of Guidance, Control, and Dynamics}, Vol. 26, No. 2, 2003, pp. 283--294.

\bibitem{Calise:98} Calise, A. J., Melamed, N., and Lee, S., ``Design and Evaluation of a Three-Dimensional Optimal Ascent Guidance Algorithm,'' \textit{Journal of Guidance, Control, and Dynamics}, Vol. 21, No. 6, 1998, pp. 867--875.

\bibitem{Baldwin:12} Baldwin, M. C., and Lu, P.,  ``{Optimal Deorbit Guidance},'' \textit{Journal of Guidance, Control, and Dynamics}, Vol. 35, No. 1, 2012, pp. 93--103.


\bibitem{Jezewski:72} Jezewski, D. J., ``{Optimal Analytical Multiburn Trajectories},'' \textit{AIAA Journal}, Vol. 10, No. 5, 1972, pp. 680--685.


\bibitem{Naidu:94} Naidu, D. S., ``Aeroassisted Orbital Transfer: Guidance and Control Strategies,'' \textit{Springer-Verlag}, New York, 1994.


\bibitem{Breakwell:63} Breakwell, J. V., Speyer, J. L., and Bryson, A. E., ``{Optimization and Control of Nonlinear Systems Using the Second Variation},'' \textit{SIAM Journal on Control}, Vol. 1, No. 2, 1963.


\bibitem{Speyer:68} Speyer, J. L., and Bryson, A. E., ``{A Neighboring Optimum Feedback Control Scheme Based on Extimated Time-to-go with Application to Reentry Flight Paths},'' \textit{AIAA Journal}, Vol. 6, No. 5, 1968.




\bibitem{Pontryagin} Pontryagin, L. S., Boltyanski, V. G., Gamkrelidze R. V., and Mishchenko E. F., ``{The Mathematical Theory of Optimal Processes (Russian)},'' \textit{English translation}: Interscience 1962.

\bibitem{Lawden:63} Lawden, D. F., ``{Optimal Trajectories for Space Navigation},'' Butterworth, London, 1963.



\bibitem{Pan:13} Pan, B., Chen, Z., Lu, P., and Gao, B., ``{Reduced Transversality Conditions for Optimal Space Trajectories},'' \textit{Journal of Guidance, Control, and Dynamics}, Vol. 36, No. 5, 2013, pp. 1289-1300.



\bibitem{Kelley:66} Kelley, H. J., Kopp, R. E., and Moyer, A. G., ``{Singular Extremals, Optimization Theory and Applications (G. Leitmann, ed.)},'' Chapter 3, Academic Press, 1966.


\bibitem{Breakwell:75} Breakwell, J. V., and Dixon, J. F.,  ``{Minimum-Fuel Rocket Trajectories Involving Intermediate-Thrust Arcs},'' \textit{Journal of Optimization Theory and Applications}, Vol. 17, No. 5/6, 1975, pp.465-479.




\bibitem{Noble:02} Noble, J. and Sch\"attler, H., ``{Sufficient Conditions for Relative Minima of Broken Extremals in Optimal Control Theory},'' \textit{Journal of Mathematical Analysis and Applications}, Vol. 269, 2002, pp.98-128.

\bibitem{Schattler:12} Sch\"{a}ttler, H. and Ledzewicz, U., ``{Geometric Optimal Control: Theory, Methods, and Examples},'' \textit{Springer}, 2012, Chaps. 5, 6.

\bibitem{Caillau:15}  Chen, Z., Caillau, J.-B., and Chitour, Y., ``{$L^1$-Minimization for Mechanical Systems},'' \textit{arXiv:1506.00569 [math.OC]}, 2015. 

\bibitem{Chen:153bp} Chen, Z., {$L^1$-Optimality Conditions for Circular Restricted Three-Body Problems},'' \textit{arXiv:1511.01816 [math.OC]}, 2015.

\bibitem{Agrachev:04} Agrachev, A. A. and Sachkov, Y. L., ``{Control Theory from the Geometric Viewpoint},'' \textit{Encyclopedia of Mathematical Sciences}, Vol. 87, Control Theory and Optimization, II. Springer-Verlag, Berlin, 2004.



\bibitem{Chen:15controllability} Chen, Z., and Chitour, Y., ``{Controllability of Keplerian Motion with Low-Thrust Control Systems},'' \textit{Radon series on Computational and Applied Mathematics}. (to be published)

\bibitem{Bryson:69} Bryson, A. E., Jr. and Ho, Y. C., ``{Applied Optimal Control: Optimization, Estimation, and Control},'' Hemisphere Publishing, Washington, D. C., 1975, Chap. 6.

\bibitem{Hull:03} Hull, D. G., ``Optimal Control Theory for Applications,'' \textit{Springer-International Edition}, New York, 2003.







\bibitem{Gergaud:06} Gergaud, J., and Haberkorn, T.,  ``{Homotopy Method for Minimum Consumption Orbital Transfer Problem},'' \textit{ESAIM: Control, Optimization and Calculus of Variations}, Vol. 12, 2006, pp. 294-310.














\bibitem{Kornhauser:72} Kornhauser, A. L., and Lion, P. M., ``{Optimal Deterministic Guidance for Bounded-Thrust Spacecrafts},'' \textit{Celestial Mechanics}, Vol. 5, 1972, pp. 261--281.

\bibitem{Mcintyre:66} Mcintyre, J. E., ``Neighboring Optimal Terminal Control with Discontinuous Forcing Functions,'' \textit{AIAA Journal}, Vol. 4, No., 1, 1966, pp. 141--148.

\bibitem{Kugelmann:90:1} Kugelmann, B., and Pesch, H. J., ``New General Guidance Method in Constrained Optimal Control, Part 1: Numerical Method,'' Journal of Optimization Theory and Applications, Vol. 67, No. 3, 1990, pp. 421--436.

\bibitem{Kugelmann:90:2} Kugelmann, B., and Pesch, H. J., ``New General Guidance Method in Constrained Optimal Control, Part 2: Application to Space Shuttle Guidance,'' Journal of Optimization Theory and Applications, Vol. 67, No. 3, 1990, pp. 437--446.



\bibitem{Pontani:15:1} Pontani, M., Cecchetti, G., and Teofilatto, P.,  ``{Variable-Time-Domain Neighboring Optimal Guidance, Part 1: Algorithm Structure},'' \textit{Journal of Optimization Theory and Applications}, Vol. 166, 2015, pp. 76--92.

\bibitem{Pontani:15:2} Pontani, M., Cecchetti, G., and Teofilatto, P.,  ``{Variable-Time-Domain Neighboring Optimal Guidance, Part 2: Application to Lunar Descent and Soft Landing},'' \textit{Journal of Optimization Theory and Applications}, Vol. 166, 2015, pp. 93--114.

\bibitem{Pontani:15:3} Pontani, M., Cecchetti, G., and Teofilatto, P.,  ``{Variable-Time-Domain Neighboring Optimal Guidance Applied to Space Trajectories},'' \textit{Acta Astronautica}, Vol. 115, 2015, pp. 102--120.


\bibitem{Afshari:09} Afshari, H. H., Novinzadeh, A. B., and Roshanian, J.,  ``{Determination of Nonlinear Optimal Feedback Law for Satellite Injection Problem Using Neighboring Optimal Control},'' \textit{American Journal of Applied Sciences}, Vol. 6, No. 3, 2009, pp. 430--438.

\bibitem{Pesch:80} Pesch, H. J., ``Neighboring Optimum Guidance of a Space-Shuttle-Orbiter-Type Vehicle,'' \textit{Journal of Guidance, Control, and Dynamics}, Vol. 3, No. 5, 1980, pp. 386--391.

\bibitem{Shafieenejad:13} Shafieenejad, I., Novinzade, A. B., and Shisheie, R., ``Analytical Mathematical Feedback Guidance Scheme for Low-Thrust Orbital Plane Change Manoeuvres,'' \textit{Mathematical and Computer Modelling}, Vol. 58, No. 11--12, 2013, pp. 1714--1726.

\bibitem{Naidu:93} Naidu, D. S., Hibey, J. L., and Charalambous, C. D., ``Neighboring Optimal Guidance for Aeroassisted Orbital Transfers,'' \textit{in Aerospace and Electronic Systems, IEEE Transaction on}, Vol, 29, No. 3, 1993, pp. 656--665.

\bibitem{Seywald:94} Seywald, H., and Cliff, E. M., ``Neighboring Optimal Control Based Feedback Law for the Advanced Launch System,'' \textit{Journal of Guidance, Control, and Dynamics}, Vol. 17, No. 6, 1994, pp. 1154--1162.

\bibitem{Mcneal:67} Mcneal, D., ``Neighboring Optimal Control of Nonlinear Systems Using Bounded Control,'' \textit{Stanford University, Aero-Astronautcs Sudaar}, No. 311, 1967.

\bibitem{Foerster:71} Foerster, R. E., and Fl\"ugge-Lotz, I., ``A Neighboring Optimal Feedback Control Scheme for Systems Using Discontinuous Control,'' \textit{Journal of Optimization Theory and Applications}, Vol. 8, No. 5, 1971, pp. 367--395.

\bibitem{Caillau:03} Caillau, J.-B., Gergaud, J., and Noailles, J.,  ``{3D Geosynchronous Transfer of a Satellite: Continuation on the Thrust},'' \textit{Journal of Optimization Theory and Applications}, Vol. 118, No. 3, 2003, pp. 541-565.





%

















\end{thebibliography}
\end{document}